\documentclass{amsart}
\usepackage{tabularray}
\usepackage[english]{babel}
\usepackage{amsmath,amssymb,amsfonts,amsthm}
\usepackage{todonotes}
\usepackage{arydshln}
\usepackage[centering]{geometry}
\usepackage{xcolor}
\usepackage[utf8]{inputenc}

\usepackage{hyperref}
\hypersetup{colorlinks   = true,
  		    citecolor    = green,
   		    linkcolor    = blue}
\usepackage{cleveref}
\usepackage{mathtools}
\usepackage{ mathrsfs }
\usepackage{environ}
\usepackage{caption}
\usepackage{systeme}
\usepackage{subcaption}
\usepackage{listings}
\usepackage{wrapfig}
\usepackage{fullpage}
\usepackage{float}

\newcommand{\R}{\ensuremath{\mathbb{R}}}

\newcommand{\SL}{\mathrm{SL}_2\mathbb{R}\times \mathrm{SL}_2\mathbb{R}}
\newcommand{\Sl}{\mathrm{SL}_2\mathbb{R}}
\def\<{\left < }
\def\>{\right >}
\def\({\left ( }
\def\){\right )}

\def\sech{\,{\rm sech\,}}

\newcommand{\xe}{\mathfrak{X}}

\newtheorem{theorem}{Theorem}[section]   
\newtheorem*{theorem*}{Theorem}          
\newtheorem{lemma}[theorem]{Lemma}
\newtheorem{proposition}[theorem]{Proposition}
\newtheorem*{genericthm*}{\thistheoremname}

\theoremstyle{definition}
\newtheorem{definition}[theorem]{Definition}
\newtheorem{corollary}[theorem]{Corollary}
\newtheorem{example}{Example}[section]

\newtheorem*{remark}{Remark}
\newcommand{\thistheoremname}{}

\theoremstyle{plain}
\newtheorem*{namedthm}{\namedthmname}
\newcounter{namedthm}

\makeatletter
\newenvironment{named}[1]
  {\def\namedthmname{#1}%
   \refstepcounter{namedthm}%
   \namedthm\def\@currentlabel{#1}}
  {\endnamedthm}
\makeatother

\title{Degenerate almost complex surfaces in the nearly K\"ahler $\mathbf{\SL}$}
\author{Kristof Dekimpe}

\thanks{The author is supported by the Research Foundation--Flanders (FWO) and the National Natural Science Foundation of China (NSFC) under collaboration project G0F2319N}
\address{{(Kristof Dekimpe\;}KU Leuven, Department of Mathematics, Celestijnenlaan 200B - Box 2400, 3001 Leuven, Belgium}
\email{kristof.dekimpe@kuleuven.be}
\begin{document}
\begin{abstract}
In this paper, we study degenerate almost complex surfaces in the semi-Riemannian nearly K\"ahler $\SL$. The geometry of these surfaces depends on the almost product structure of the ambient space and one can distinguish two distinct cases. The geometry of these surfaces is influenced by the almost product structure of the ambient space, leading to two distinct cases. The first case arises when the tangent bundle of the surface is preserved under the almost product structure, while the second case occurs when the tangent bundle of the surface is not invariant under this structure. In both cases, we obtain a complete and explicit classification.
\end{abstract}
\maketitle
\section*{Introduction}
Perhaps the most important class of almost Hermitian manifolds, (semi-) Riemannian manifolds endowed with a compatible almost complex structure $J$, are the K\"ahler manifolds, for which the almost complex structure $J$ is parallel with respect to the Levi-Civita connection $\nabla$ \cite{Yao}. These manifolds admit compatible complex, symplectic and Riemannian structures, hence the broad interest in them from a geometric point of view. Nearly Kähler manifolds are a relaxed version of Kähler manifolds, where the Kähler condition is not necessarily satisfied. Instead, the tensor field $\nabla J$ must be skew-symmetric \cite{book}. When these manifolds do not meet the Kähler condition, they are referred to as strict nearly Kähler manifolds. An important distinction arising from the relaxation of the Kähler condition is that the fundamental two-form on a nearly Kähler manifold is not necessarily closed, which means they may not be symplectic \cite{gray2}. 

Specific studies of these manifolds were initiated in the 1970s by Gray \cite{gray1}, and Nagy provided a classification result for strict nearly Kähler manifolds in 2002 \cite{nagy}. Nagy showed that such manifolds are products of 3-symmetric spaces, twistor spaces over quaternionic spaces with positive scalar curvature, and six-dimensional nearly Kähler manifolds. The six-dimensional nearly Kähler manifolds are particularly interesting because they are the lowest-dimensional manifolds in which strict nearly Kähler structures can exist. Additionally, they are the only six-dimensional manifolds that admit Killing spinors \cite{Ines}. Butruille later demonstrated that the only \textit{homogeneous} six-dimensional Riemannian nearly Kähler manifolds are the six-sphere $S^6$, the product manifold $S^3\times S^3$ (not with the usual product metric), the complex projective space $\mathbb{C}P^3$ (not with the Fubini-Study metric) and the flag manifold $F_{2,1}(\mathbb{C}^3)$ \cite{butruille}. 
More recently, Foscolo and Haskin discovered examples of non-homogeneous nearly Kähler structures on $S^6$ and $S^3\times S^3$, highlighting the necessity of homogeneity in the classification \cite{foscolo}. While the semi-Riemannian analogues of the manifolds appearing in Butruille's list have been studied, a complete classification in the semi-Riemannian case remains elusive. In fact, there are examples of semi-Riemannian nearly Kähler manifolds without a Riemannian counterpart \cite{Schafer}.

In this paper, we focus on the semi-Riemannian nearly Kähler $\SL$, which serves as the semi-Riemannian analogue of the nearly K\"ahler $S^3 \times S^3$ \cite{moruz}. Two important classes of submanifolds in almost Hermitian manifolds are the almost complex and the totally real submanifolds. The former refers to submanifolds where the almost complex structure $J$ preserves every tangent space, while the latter refers to submanifolds where $J$ maps every tangent space into the normal space.

In a recent paper by Ghandour and Vrancken \cite{Ghandour} the classification of semi-Riemannian totally geodesic almost complex surfaces in the nearly Kähler $\SL$ was achieved using techniques from the classification in the nearly Kähler $S^3\times S^3$ \cite{bolton}. These papers introduced an almost product structure $P$ on $\SL$, which is an involutive and symmetric endomorphism that anti-commutes with the almost complex structure $J$. This almost product structure plays a significant role in studying submanifolds as it partially determines their intrinsic and extrinsic geometry by its behavior when restricted to the tangent spaces of the submanifolds. We present a generalization of the results obtained by Ghandour and Vrancken in this paper, as we classify all degenerate almost complex surfaces in the semi-Riemannian nearly Kähler manifold $\SL$. It is important to emphasize that degenerate surfaces can only exist in semi-Riemannian manifolds and not in Riemannian manifolds. However, it is worth noting that the notion of degenerate almost complex surfaces remains well-defined, as the almost complex condition is independent of the induced metric. To simplify the analysis of these surfaces we introduce the distribution $\mathcal{D}$ on a degenerate almost complex surface $\Sigma$, defined as 
\begin{align}\label{def: distrD}
	\mathcal{D}=T\Sigma \oplus PT\Sigma,
\end{align}
with $P$ the almost product structure of the nearly K\"ahler $\SL$. One can thus consider two distinct cases. In the first case, the tangent bundle of the surface is preserved under the almost product structure, resulting in a two-dimensional distribution $\mathcal{D}$. The second case arises when the tangent bundle of the surface is not invariant under the almost product structure, leading to a four-dimensional distribution $\mathcal{D}$.
The first main result of this paper is the classification of degenerate almost complex surfaces when the almost product structure $P$ preserves the tangent bundle of the surface.  
\begin{named}{Main Theorem 1}\label{Main theorem 1}
Let $\Sigma$ be an immersed, degenerate almost complex surface in the nearly K\"ahler \\${(\SL, J, g)}$ for which the almost product structure $P$ preserves the tangent bundle. Then the surface $\Sigma$ is locally congruent to an open part of the immersion 
\begin{align*}
	F: \mathbb{R}^2\rightarrow\SL: (s, t)\mapsto\left(\begin{pmatrix}
		1 & 2s \\
		0 & 1\\
	\end{pmatrix},\begin{pmatrix}
		1 & 2t \\
		0 & 1\\
	\end{pmatrix}\right).
\end{align*}
\end{named}
One can also consider the second case, where the tangent bundle of the surface is not invariant under the almost product structure $P$. Notably, the four-dimensional distribution $\mathcal{D}$ cannot be degenerate since the maximal dimension of a degenerate six-dimensional semi-Riemannian manifold with an index of two is precisely two. Consequently, it becomes always possible to locally define a unique (up to sign) vector field $X$ on the surface $\Sigma$ that satisfies the following conditions:
\begin{align}\label{eq:VFX}
	g(X, PX) = 1 \quad \text{and} \quad g(X, JPX) = 0.
\end{align}
This vector field can then be employed to introduce a constant angle function that locally determines these surfaces. This is demonstrated in \ref{Main theorem 2}.
\begin{named}{Main Theorem 2}\label{Main theorem 2}
Let $\Sigma$ be an immersed, degenerate almost complex surface in the nearly K\"ahler \\${(\SL, J, g)}$ for which the almost product structure $P$ does not preserve the tangent bundle. Let $X$ be the locally defined, unique (up to a sign) vector field satisfying Equation (\ref{eq:VFX}) and $\nabla$ the Levi-Civita connection of the nearly K\"ahler metric $g$. The following statements then hold:
\begin{itemize}
	\item[(i)] There exists a constant angle function $\phi$ on $\Sigma$ that satisfies
	\begin{align*}
		\cos\phi=g(\nabla_X X, (\nabla_X J)PX).
	\end{align*} 
	\item[(ii)]If two immersions $F, G:\Sigma\rightarrow:\SL$ are degenerate almost complex surfaces with the same constant angle function $\phi$, then the immersions locally coincide.
\end{itemize}
\end{named}
If we consider an immersed, degenerate almost complex surface 
\begin{align*}
F:\Sigma\rightarrow\SL:(x, y)\mapsto(p(x, y), q(x, y)),
\end{align*} in the nearly K\"ahler $\SL$, then it is also possible to analyze the immersions $p,q:\Sigma\rightarrow \SL$ in the respective semi-Riemannian factors $\Sl$. One of the implications of \ref{Main theorem 2} is that these immersions are Lorentzian isoparametric surfaces in $\Sl$. By examining these isoparametric surfaces, we can achieve a complete classification of all degenerate almost complex surfaces in $\SL$ where the almost product structure $P$ does not preserve the tangent bundle. This will be demonstrated in \ref{theorem:classifying result}.

This paper is organized as follows: we begin with some preliminaries, focusing on defining nearly Kähler manifolds and their relevant submanifolds.
In Section \ref{section:SL}, we define the nearly Kähler structure on the product manifold $\SL$ and derive general properties of this space. We introduce an almost product structure $P$, which allows us to describe the geometry of the semi-Riemannian nearly Kähler $\SL$ in terms of the nearly Kähler metric $g$, the almost complex structure $J$, and the almost product structure $P$.
Section \ref{section:proofMT1} presents a proof for \ref{Main theorem 1}, classifying all degenerate almost complex surface for which $P$ preserves the tangent bundle. In Section \ref{section:proofMT2}, we explore the case where $P$ does not preserve the tangent bundle of the surface. This leads us to prove \ref{Main theorem 2}, which demonstrates that a locally determined constant angle function $\phi$ characterizes such surfaces. Section \ref{section:Isoparam} focuses on the analysis of the components of an immersion $F$ of a degenerate almost complex surface in the nearly K\"ahler $\SL$ into the separate semi-Riemannian factors $\Sl$. We establish that these immersions are Lorentzian isoparametric surfaces in $\Sl$ and provide a classification based on the introduced angle function $\phi$. In Section \ref{section:classifyingresult} we provide examples of degenerate almost complex surfaces for which $P$ does not preserve the tangent bundle and prove \ref{theorem:classifying result}, obtaining a complete classification of these surfaces in the nearly K\"ahler $\SL$.

\section{Preliminaries}
In this section, we will provide a review of fundamental definitions, properties, and formulas related to nearly semi-Riemannian submanifolds and general nearly Kähler manifolds. In Section \ref{section:SL}, we will elaborate on the nearly K\"ahler structure on $\SL$ and introduce an almost product structure $P$, which allows us to describe the manifold's geometry in terms of $P$ and the almost complex structure $J$.
\subsection{Semi-Riemannian submanifolds}
A semi-Riemannian manifold is a manifold with a metric $g$, which is not required to be positive definite, but merely non-degenerate. Consequently, one can distinguish three types of tangent vectors to such a manifold: a vector $v$ is \textit{spacelike} if $g(v,v) > 0$ or $v=0$, \textit{timelike} if $g(v,v) < 0$ and \textit{lightlike} or \textit{null} if $g(v,v) = 0$ but $v \neq 0$. We call a semi-Riemannian submanifold $M$ of an ambient semi-Riemannian manifold $(\tilde{M},g)$ \textit{spacelike} if the induced metric on $M$ is Riemannian and \textit{Lorentzian} if the index $\nu$ of the induced metric on $M$ satisfies $\nu=1$. If $M$ is an immersed submanifold in a semi-Riemannian manifold $(\tilde{M},g)$ and if the metric $g$ restricted to the submanifold $M$ is completely degenerate, then one calls $M$ a \textit{degenerate submanifold}. 

We now consider an immersed semi-Riemannian submanifold $F:M\rightarrow \tilde{M}$ of an ambient semi-Riemannian manifold $(\tilde{M},g)$ and the immersion $F$ will be omitted from the notation, as locally every immersion is an embedding. The following formulas and notations will be used throughout this paper. Denote by $\nabla$ and $\tilde{\nabla}$ the Levi-Civita connections on $M$ and $\tilde{M}$, respectively.  The formulas of Gauss and Weingarten respectively state that
\begin{align}
& \widetilde{\nabla}_X Y=\nabla_X Y + h(X,Y), \label{Gaussformula}\\
& \widetilde{\nabla}_X \xi = -A_{\xi}X + \nabla^{\perp}_X \xi\label{Weingarten}
\end{align}
for vector fields $X$ and $Y$ tangent to $M$ and a vector field $\xi$ normal to $M$. Here, $h$ is the second fundamental form, taking values in the normal bundle, $A_{\xi}$ is the shape operator associated to the normal vector field $\xi$ and $\nabla^{\perp}$ is the normal connection. The mean curvature vector field $H$ is defined as the averaged trace of the second fundamental form $h$. A submanifold $M$ is a \textit{minimal submanifold} if the mean curvature vector field $H$ is everywhere zero.
The second fundamental form and the shape operator are related by
\begin{align*}
g(h(X,Y),\xi)=g(A_\xi X,Y),
\end{align*}
while the normal connection is used in defining the covariant derivative of the second fundamental form as
\begin{align*}
(\overline{\nabla}h)(X,Y,Z)=\nabla^\perp_X h(Y,Z)-h(\nabla_XY)-h(Y,\nabla_x Z),
\end{align*}
for all vector fields $X$, $Y$ and $Z$ tangent to $M$.
Denote by $X^T$ and $X^\perp$ the tangent and normal part of a vector field $X$ with respect to the submanifold $M$. The equations of Gauss, Codazzi and Ricci will also be used in this paper and can now respectively be stated as 
\begin{align}
(\tilde{R}(X,Y)Z)^T&=R(X,Y)Z +A_{h(X,Z)}Y- A_{h(Y,Z)}X \label{Gaussequation},\\
(\tilde{R}(X,Y)Z)^\perp&=(\overline{\nabla}h)(X,Y,Z)-(\overline{\nabla}h)(Y,X,Z)\label{Codazzi},\\
(\tilde{R}(X,Y)\xi)^\perp&=R^\perp(X,Y)\xi+h(A_\xi X,Y)-h(X, A_\xi Y),\label{Ricci}
\end{align}
for tangent vector fields $X,Y,Z$ and normal vector field $\xi$. Here, $R$ and $\tilde{R}$ are the curvature tensors of $M$ and $\tilde{M}$, respectively, while $R^\perp$ is the normal curvature tensor. To conclude this section, we recall the definition of the semi-Riemannian hyperbolic space $H^n_s$. Let $\mathbb{R}^n_s$ denote $\R^n = \{(x_1,\ldots,x_n) \ | \ x_1,\ldots,x_n \in~\R \}$ equipped with the inner product
\begin{align*}
\langle (x_1,\ldots,x_n) , (y_1,\ldots,y_n) \rangle = -x_1 y_1 - \ldots - x_s y_s + x_{s+1} y_{s+1} + \ldots + x_n y_n.
\end{align*}
For $s=0$, the space $\R^n_0 = \R^n$ is just the Euclidean space of dimension $n$ and for $s>1$, we call $\R^n_s$ the semi-Euclidean space of dimension $n$ and index $s$. It is a flat manifold, i.e., a semi-Riemannian manifold with constant sectional curvature $0$. We now define
\begin{align*}
H^n_s(c) =\left\{x\in\mathbb{R}^{n+1}_{s+1} \ | \ \left\langle x,x \right\rangle = 1/c \right\} \mbox{ for }c<0.
\end{align*}
The manifold $H^n_s(c)$, equipped with the induced metric from $\R^{n+1}_{s+1}$, is a complete semi-Riemannian manifold with constant sectional curvature $c$. Together with the semi-Euclidean $\mathbb{R}^n_s$ and the semi-Riemannian spheres $S^n_s(c)$, it forms the semi-Riemannian space forms, often denoted as $\tilde{M}^n_s(c)$. When the curvature $c$ satisfies $c=\pm1$, we use the shorthand $S^n_s(c)$ and $H^n_s$. If the index $s$ is equal to zero or one, then the space forms are called Riemannian space forms or Lorentzian, respectively. In particular, the three-dimensional Lorentzian real space forms  $\mathbb{R}^3_1$, $S^3_1$ and $H^3_1$ are known as the \textit{Minkowksi spacetime}, the \textit{de Sitter spacetime} and the \textit{anti-de Sitter} spacetime. In Section \ref{section:Correspondence} we will construct a local isometry between the anti-de Sitter space $H^3_1$ and the semi-Riemannian $\Sl$. We also present the definition of an isoparametric hypersurface of semi-Riemannian space form $\tilde{M}^n_s(c)$, with $c$ the constant sectional curvature.
\begin{definition}
Let $M$ be a semi-Riemannian hypersurface of a semi-Riemannian real space form $\tilde{M}^n_s(c)$. Then $M$ is \textit{isoparametric} if it has constant principal curvatures. 
\end{definition}
\subsection{Nearly K\"ahler manifolds}
In this paper we will study the degenerate almost complex surfaces of the nearly K\"ahler $\SL$, which is a six-dimensional strict nearly K\"ahler manifold. We first introduce the notion of a nearly K\"ahler manifold and present some useful lemmas and formulas. 
\begin{definition}
A differentiable manifold $M$ is said to be \textit{almost complex} if there exist a differentiable (1,1) tensor field $J$ such that $J^2=-Id$. An \textit{almost Hermitian} manifold is an almost complex manifold $(M, J)$, endowed with a compatible semi-Riemannian metric $g$, meaning that for all vector fields $X, Y$ on $M$ it holds that
\begin{align*}
g(JX, JY)=g(X, Y).
\end{align*}
\end{definition}
Notice that an almost complex manifold is necessarily even-dimensional and orientable and that it is always possible to endow an almost complex manifold with a compatible Riemannian metric. An important class of almost Hermitian manifolds are the K\"ahler manifolds, given by the following definition.
\begin{definition}
An almost Hermitian manifold $(M,J, g)$ is called a \textit{K\"ahler manifold} if for all vector fields $X, Y$ on $M$ it holds that 
\begin{align*}
(\nabla_X J)Y=0,
\end{align*}
with $\nabla$ the Levi-Civita connection of $M$.
\end{definition}
The condition that $\nabla J=0$ means that the almost complex structure $J$ is parallel, thus that is behaves the same on each tangent space of $M$. One can easily show that a K\"ahler manifold has compatible complex, symplectic and Riemannian structures \cite{Yao}. Relaxing the K\"ahler condition then yields the notion of a nearly K\"ahler manifold.
\begin{definition}\label{def:NK}
A \textit{nearly K\"ahler} manifold is an almost Hermitian manifold $(M, J, g)$ on which the almost complex structure $J$ satisfies
\begin{align*}
(\nabla_X J)X&=\nabla_X(JX)-J\nabla_X X=0
\end{align*}
for any tangent vector field $X\in \xe(M)$, with $\nabla$ the Levi-Civita connection on $M$. This is equivalent with saying that the tensor field $G$, defined as $G(X, Y)=(\nabla_X J)Y$, is skew-symmetric for all $X, Y\in\xe(M)$.
\end{definition}
As the tensor $G$ is essential in any study of nearly K\"ahler manifolds, we will first mention some of its elementary properties \cite{Dioos}.
\begin{proposition}[\cite{Dioos}]\label{prop: G}
Let $M$ be a nearly K\"ahler manifold with metric $g$ and almost complex structure $J$. The following expressions hold for the tensor $G$ and the Riemann tensor $R$:
\begin{align*}
&G(X,Y)+G(Y, X)=0,\\
&G(X, JY)+JG(X, Y)=0,\\
&G(X, Y)+G(JX, JY)=0,\\
&g(G(X, Y), Z)+g(G(X, Z), Y)=0,\\
&g(G(X, Y), G(X, Y))=R(X, Y, Y, X)-R(X, Y, JY, JX),\\
&g(G(X, Y), G(W, Z))=R(X, Y, Z, W)-R(X, Y, JZ, JW),\\
&R(W, X, Y, Z)=R(JW, JX, JY, JZ),\\
&2g((\nabla G)(W, X, Y), Z)= \underset{XYZ}{\mathfrak{S}}g(G(W, X), JG(Y, Z)),
\end{align*}
where $\mathfrak{S}$ stands for the cyclic sum over the vector fields $X, Y$ and $Z$.
\end{proposition}
One then immediately gets the following corollary. 
\begin{corollary}
Every two or four dimensional nearly K\"ahler manifold is automatically K\"ahler.
\end{corollary}
One of the most natural types of submanifolds to consider in almost Hermitian manifolds are the almost complex submanifolds.
\begin{definition}
Let $M$ be a submanifold of an almost Hermitian manifold $(\tilde{M}, J, g)$. Then $M$ is an \textit{almost complex submanifold} if the almost complex structure $J$ preserves the tangent space of $M$, thus if $J(TM)=TM$.
\end{definition}
Notice that, as this definition is independent on the induced metric on the submanifold $M$, that one can in fact consider degenerate almost complex surfaces in semi-Riemannian nearly K\"ahler manifolds. 

\section{The nearly K\"ahler $\SL$}\label{section:SL}
In this section, we introduce the nearly K\"ahler structure on $\SL$ as presented in \cite{Ghandour}. We begin by equipping $\Sl$, the space of real $2\times 2$ matrices with determinant 1, with a semi-Riemannian metric.
\subsection{The semi-Riemannian $\Sl$}
Consider the non-degenerate, indefinite inner product $\left\langle\, , \,\right\rangle$ on $\mathbb{R}^4$, given by
\begin{align*}
	\left\langle(x_1, x_2, x_3, x_4), (y_1, y_2, y_3, y_4)\right\rangle=-\frac{1}{2}(x_1y_4-x_2y_3-x_3y_2+x_4y_1), 
\end{align*} 
with $(x_1, x_2,x_3, x_4),  (y_1, y_2, y_3, y_4)\in\mathbb{R}^4$. One can identify the space of $2\times2$ real matrices $\mathbb{R}^{2\times 2}$ with $\mathbb{R}^4$, thus one can see the above inner product on $\mathbb{R}^{2\times 2}$ as 
\begin{align*}
	\left\langle A,B \right\rangle=-\frac{1}{2}\text{Trace}((\text{adj}A)^TB),
\end{align*}
with $A, B\in\mathbb{R}^{2\times 2}$. This yields an alternative way to describe the space $\Sl$ as 
\begin{align*}
	\Sl=\left\{A\in\mathbb{R}^{2\times 2}\,|\,\left\langle A, A\right\rangle=-1 \right\}.
\end{align*}
As usual, the restriction of the inner product $\left\langle\, , \,\right\rangle$ to the tangent spaces of $\Sl$ will also be denoted by $\left\langle\, , \,\right\rangle$. As $\Sl$ is in fact a Lie group, it has an associated Lie algebra $\mathfrak{sl}_2\mathbb{R}$, which is the space of real traceless $2\times2$ matrices. Hence, the tangent space at a point $A\in\Sl$ can be described as
\begin{align*}
	T_A\Sl=\left\{A\alpha\, |\,\alpha\in\mathfrak{sl}_2\mathbb{R} \right\}.
\end{align*}
The Lie algebra $\mathfrak{sl}_2\mathbb{R}$, equipped with the inner product $\left\langle\, , \,\right\rangle$, has an semi-orthonormal basis consisting of the split-quaternions $\mathfrak{i}, \mathfrak{j}, \mathfrak{k}$, given by
\begin{align}\label{CAH:def:splitquaternions}
	\mathfrak{i}=\begin{pmatrix}
		1 & 0 \\
		0 & -1 \\
	\end{pmatrix},\quad{}
	\mathfrak{j}=\begin{pmatrix}
		0 & 1 \\
		1 & 0 \\
	\end{pmatrix},\quad{}
	\mathfrak{k}=\begin{pmatrix}
		0 & 1 \\
		-1 & 0 \\
	\end{pmatrix},
\end{align}
with inner products
\begin{align*}
	\left\langle \mathfrak{i}, \mathfrak{i}\right\rangle=1,\; \left\langle \mathfrak{j}, \mathfrak{j}\right\rangle=1,\; \left\langle \mathfrak{k}, \mathfrak{k}\right\rangle=-1. 
\end{align*}
Thus one can define tangent vector fields $X_1, X_2, X_3$ on $\Sl$ by 
\begin{align}\label{CAH:def:vectorfieldsX}
	X_1(A)=A\mathfrak{1},\quad{} X_2(A)=A\mathfrak{j}, \quad{}X_3(A)=A\mathfrak{k}
\end{align}
for each point $A\in\Sl$. These vector fields clearly form a semi-orthonormal basis of $T_A\Sl$, ensuring that $(\Sl, \left\langle\; ,\; \right\rangle)$ is a three-dimensional semi-Riemannian manifold of index 1. A straightforward calculation shows that the metric $\left\langle \, , \,\right\rangle$ is a bi-invariant metric and we will see that $(\Sl, \left\langle\; ,\;\right\rangle)$ is in fact locally isometric to the anti-de Sitter space $H^3_1$.
\subsubsection*{\textbf{Correspondence between $\mathbf{\Sl}$ and $\mathbf{H^3_1}$ }}\label{section:Correspondence}
One can straightforwardly construct a local isometry between $(H^3_1, \left\langle\, ,\,\right\rangle)$ and $(\Sl, \left\langle\, , \,\right\rangle)$ as  
\begin{align}\label{eq:corr1}
	\mathfrak{f}_1:H^3_1\rightarrow \Sl: (x_1, x_2, x_3, x_4)\mapsto\begin{pmatrix}
		x_1-x_3 & x_4-x_2\\
		x_4+x_2 & x_1+x_3
	\end{pmatrix}.
\end{align}
There is an alternative way of describing the metric $\left\langle\, ,\,\right\rangle$ on $\mathbb{R}^4_2$, which we will denote by $\left\langle\, , \,\right\rangle_*$ , as 
\begin{align}\label{metricH}
	\left\langle(x_1, x_2, x_3,x_4),(y_1, y_2, y_3, y_4)\right\rangle_{*}&=\frac{1}{2}\,(x_1y_4+x_2y_3+x_3y_2+x_4y_1),
\end{align}
for $(x_1, x_2, x_3, x_4),(y_1, y_2, y_3, y_4)\in\mathbb{R}^4$. Thus an equivalent way of describing the anti-de Sitter space is given by 
\begin{align*}
	H^3_1=\left\{x\in\mathbb{R}^4_2\,|\,\left\langle x, x\right\rangle_*=-1 \right\}.
\end{align*} 
One can then also construct an local isometry between $(H^3_1, \left\langle\, ,\,\right\rangle_*)$ and $(\Sl, \left\langle\, , \,\right\rangle)$ as  
\begin{align}\label{eq:corr2}
	\mathfrak{f}_2:H^3_1\rightarrow \Sl: (x_1, x_2, x_3, x_4)\mapsto\begin{pmatrix}
		x_1 & x_2\\
		x_3 & -x_4
	\end{pmatrix}.
\end{align}
\subsection{Nearly K\"ahler structure on $\SL$}
We will now consider the product space $\SL$, whose tangent space at a point $(A, B)$ can be described using the natural identification $T_{(A,B)}(\SL)\cong T_A \Sl\,\oplus\, T_B\Sl$. The tangent space of $\SL$ at a point $(A,B)$ is thus given by 
\begin{align*}
	T_{(A,B)}\SL&=\left\{(A\alpha, B\beta)\,|\, \alpha, \beta\in\mathfrak{sl}_2\mathbb{R}\right\}.
\end{align*}
This implies that one can write a tangent vector at $(A, B)$ as $Z(A, B)=(U(A,B), V(A,B))$ or simply as $Z=(U,V)$ for $U$ and $V$ vector fields on $\Sl$.
The usual product metric on $\SL$ is then, for $Z=(U, V), Z'=(U', V')$ vector fields on $\SL$, constructed from the metric $\left\langle\, ,\,\right\rangle$ on $\Sl$ as
\begin{align*}
	\left\langle Z, Z'\right\rangle_{\SL}&=\left\langle U, U'\right\rangle+\left\langle V, V'\right\rangle.
\end{align*}
Remark that from now on the product metric $\left\langle\, , \,\right\rangle_{\SL}$  will also simply be denoted by $\left\langle\, , \,\right\rangle$. In an analogous way to the definition of the almost complex structure on $S^3\times S^3$ \cite{moruz}, one can also introduce an almost complex structure $J$ on $\SL$, which is for an arbitrary tangent vector $(A\alpha, B\beta)\in T_{(A, B)}\SL$ given by  
\begin{align}\label{def:J}
	J(A\alpha, B\beta)=\frac{1}{\sqrt{3}}(A(\alpha-2\beta),B(2\alpha-\beta)).
\end{align}
The nearly K\"ahler metric $g$ on $\SL$ is, up to a rescaling, the usual Hermitian metric associated to the usual product metric and is, for $Z, Z'$ vector fields on $\SL$, given by 
\begin{align*}
	g(Z, Z')&=\frac{1}{4}(\left\langle Z, Z'\right\rangle+\left\langle JZ, JZ'\right\rangle), 
\end{align*}
which is clearly compatible with the almost complex structure $J$. The nearly K\"ahler metric can be written explicitly for tangent vectors $Z=(A\alpha, B\beta), Z'=(A\gamma, B\delta)\in T_{(A, B)}\SL$ as
\begin{align*}
	g((A\alpha, B\beta), (A\gamma, B\delta))&=\frac{2}{3}\left\langle(A\alpha, B\beta), (A\gamma, B\delta)\right\rangle-\frac{1}{3}\left\langle(A\beta, B\alpha), (A\gamma, B\delta)\right\rangle.
\end{align*}
Note that the metric can be traced back to the submersion 
\begin{align*}
	\pi: \SL\times SL(2, \mathbb{R})\rightarrow \SL: (A, B, C)\mapsto (AC^{-1}, BC^{-1}),
\end{align*}
making it a semi-Riemannian submersion \cite{moruz}. The following lemma shows an explicit formula for the tensor $G$, given in Definition \ref{def:NK} on $\SL$.
\begin{lemma}
	Let $X=(A\alpha, B\beta), Y=(A\gamma, B\delta)\in T_{(A, B)}\SL$. Then
	\begin{align*}
		G(X,Y)=\frac{2}{3\sqrt{3}}(A(-\alpha\times\gamma-\alpha\times\delta+\gamma\times\beta+2\beta\times\delta),B(-2\alpha\times\gamma+\alpha\times\delta-\gamma\times\beta+\beta\times\delta)),
	\end{align*}
	where $\times$ is defined on $\mathfrak{sl}_2\mathbb{R}$ as $\alpha\times\beta=\frac{1}{2}(\alpha\beta-\beta\alpha)$.
\end{lemma}
As $\SL$ is a product space, it is natural to introduce an almost product structure $P$, which will be compatible with the nearly K\"ahler metric. This (1, 1) tensor field is pointwise defined as
\begin{align}\label{def:almostproduct}
	PZ=P(A\alpha, B\beta)=(A\beta, B\alpha),
\end{align}
for all $Z=(A\alpha, B\beta)\in T_{(A,B)}\SL$.
The following lemma summarizes some elementary properties of the almost product structure.
\begin{lemma}
	The almost product structure $P$ on the nearly K\"ahler $\SL$ satisfies the following properties:
	\begin{align}
		P^2&=Id,\; \text{i.e. $P$ is involutive}\label{eq:Pid},\\
		PJ&=-JP,\; \text{i.e $P$ and $J$ anti-commute},\\
		g(PX, PY)&=g(X,Y),\; \text{i.e. $P$ is compatible with $g$}\label{eq:Pcomp},\\
		g(PX, Y)&=g(X, PY),\; \text{i.e. $P$ is symmetric with respect to the metric $g$}\label{eq:Psymm}.
	\end{align}
\end{lemma}
A straightforward computation shows that the tensor field $\nabla P$ does not vanish identically, thus the $(1, 1)$ tensor field $P$ is not parallel with respect to the nearly K\"ahler Levi-Civita connection. However, the following relations do hold for the almost product structure.
\begin{lemma}\label{lem:propP}
	For vector fields $X, Y$ on $\SL$ the following relations hold:
	\begin{align*}
		&PG(X,Y)+G(PX, PY)=0\\
		&(\nabla_X P)JY=J(\nabla_X P)Y,\\
		&G(X, PY)+PG(X, Y)=-2J(\nabla_X P)Y,\\
		&(\nabla_X P)PY+P(\nabla_X P)Y=0,\\
		&(\nabla_X P)Y+(\nabla_{PX}P)Y=0.
	\end{align*}
\end{lemma}
A long, but straightforward computation shows that the Riemann curvature tensor $R$ of the nearly K\"ahler $(\SL, J, g)$ can be written in terms of the metric $g$, the almost complex structure $J$ and the almost product structure $P$, as shown in the following lemma \cite{Ghandour}.
\begin{lemma}[\cite{Ghandour}]\label{lem:R}
	The Riemann curvature tensor on the nearly K\"ahler $\SL$ is given by
	\begin{align}
		R(X,Y,Z)=&-\frac{5}{6}(g(Y,Z)X-g(X,Z)Y)\\
		&-\frac{1}{6}(g(JY,Z)JX-g(JX,Z)JY-2g(JX,Y)JZ)\\
		&-\frac{2}{3}(g(PY,Z)PX-g(PX,Z)PY+g(JPY, Z)JPX-g(JPX, Z)JPY).
	\end{align}
\end{lemma}
The expression for the curvature tensor, together with Proposition \ref{prop: G}, then immediately shows the following lemma. 
\begin{lemma}\label{lem:lengthG}
	The tensor $G$ on the nearly K\"ahler $\SL$ satisfies the following expressions:
	\begin{align*}
		&g(G(X, Y), G(Z, W))=-\frac{2}{3}(g(X, Z)g(Y, W)-g(X, W)g(Y, Z)+g(JX, Z)(g(JW, Y)-g(JX, W)g(JZ, Y)),\\
		&G(X, G(Y, Z))=-\frac{2}{3}(g(X, Z)Y-g(X, Y)Z+g(JX, Z)JY-g(JX, Y)JZ), \\
		&(\nabla G)(X, Y, Z)=-\frac{2}{3}(g(X, Z)JY-g(X, Y)JZ-g(JY, Z)X),
	\end{align*}
	for all vector fields $W, X, Y, Z$ on $\SL$.
\end{lemma}
Recall that one could also consider the product manifold \\${(\SL, \left\langle\, , \,\right\rangle)}$, whose geometry is intimately  related to that of the nearly K\"ahler $(\SL, g)$. To understand this relation, it is useful to introduce the usual product structure $Q$, pointwise defined as
\begin{align}\label{def:Q}
	QZ=Q(A\alpha, B\beta)=(-A\alpha, B\beta),
\end{align}
for all $Z=(A\alpha, B\beta)\in T_{(A,B)}\SL$.
The almost product structure $Q$ can then be expressed in terms of the almost product structure $P$ and vice versa for all vector fields $Z$ on $\SL$: 
\begin{align*}
	QZ&=-\frac{1}{\sqrt{3}}(2PJZ-JZ),\\
	PZ&=\frac{1}{2}(Z+\sqrt{3}JQZ).
\end{align*}
Using these equations one can write the semi-Euclidean metric $\left\langle\, ,\,\right\rangle$ in  terms of the K\"ahler metric $g$ by:
\begin{align}\label{def:Euclmetric}
	\left\langle Z, Z'\right\rangle=2g(Z, Z')+g(Z, PZ'),
\end{align} 
for all vector fields $Z, Z'$ on $\SL$. It is now possible to show the relation between the Levi-Civita connections $\nabla^E$ and $\nabla$ of the Euclidean metric $\left\langle\, ,\,\right\rangle$ and nearly K\"ahler metric $g$, respectively.
\begin{lemma}
	The relation between the Levi-Civita connection $\nabla$ of the nearly K\"{a}hler metric $g$ and that of the usual product metric $\left\langle\, .\, ,\, .\right\rangle$, denoted by $\nabla^E$, on $\SL$ is given by
	\begin{align}\label{CAH:equation:Euclconnection} 
		\nabla_X^E Y=\nabla_X Y+\frac{1}{2}(JG(X,PY)+JG(Y,PX)).
	\end{align}
\end{lemma}
It is sometimes also useful to relate the semi-Euclidean connection $D$ of $\mathbb{R}^8_4$ with the semi-Euclidean connection $\nabla^E$ on $\SL$.
\begin{lemma}\label{lem:Euclcon}
	Let $X, Y$ be vector fields on $\SL$. At a point $(A, B)\in\SL$, the relationship between the Levi-Civita connection $D$ of the semi-Riemannian $\mathbb{R}^8_4$ and the semi-Euclidean connection $\nabla^E$ on $(\SL, \left\langle;\right\rangle)$ is given by
	\begin{align*}
		D_X Y=\nabla_X^E Y+\frac{1}{2}\left\langle X, Y\right\rangle (A,B)+\frac{1}{2}\left\langle X, QY\right\rangle (-A,B).
	\end{align*}
\end{lemma}
As the nearly K\"ahler is a semi-Riemannian manifold, one can look for isometries that are in some sense compatible with the almost complex structure $J$ and the almost product structure $P$, as shown in \cite{moruz}.

\begin{lemma}\label{lem:isom}
	The maps $\mathcal{F}_1$  and $\mathcal{F}_{ABC}$, defined as
	\begin{align*}
		&\mathcal{F}_1:\SL\rightarrow\SL: (p,q)\mapsto (q,p),\\
		&\mathcal{F}_{ABC}:\SL\rightarrow\SL: (p,q)\mapsto(ApC, BqC),
	\end{align*} 
	are isometries. Moreover, the following relations hold with respect to the almost complex structure $J$ and almost product structure $P$:
	\begin{alignat*}{5}
		& &&d\mathcal{F}_1\circ J = -J\circ d\mathcal{F}_1,\;\quad{} &&d\mathcal{F}_1\circ P= P\circ\mathcal{F}_1,\\
		& &&d\mathcal{F}_{ABC}\circ J = J\circ d\mathcal{F}_{ABC},\;\quad{} &&d\mathcal{F}_{ABC}\circ P= P\circ\mathcal{F}_{ABC}.
	\end{alignat*}
\end{lemma}

\section{Degenerate almost complex surfaces with two-dimensional distribution $\mathcal{D}$}\label{section:proofMT1}
In this section, we examine the case where the distribution $\mathcal{D}$, defined in Equation (\ref{def: distrD}), on a degenerate almost complex surface $\Sigma$ is two-dimensional. Consequently, the almost product structure $P$ preserves the tangent bundle of this surface. Our objective is now to prove \ref{Main theorem 1}.
\begin{proof}[Proof of Main Theorem 1]
 As the surface $\Sigma$ in $\SL$ is almost complex, one can consider a local frame $\left\{X, JX \right\}$ on $\Sigma$. Due to $\mathcal{D}$ being two-dimensional, one has that the almost product structure $P$ and the product $JP$ leave the tangent spaces invariant. It is thus possible to choose the frame $\left\{X, JX \right\}$ in such a way that the almost product structure with respect to this frame is given by
\begin{align*}
PX=X,\; PJX=-JX.
\end{align*}
As $\Sigma$ is an immersed surface in $\SL$, one can explicitly write the immersion as
\begin{align*}
F:\Sigma\rightarrow \SL:(x,y)\mapsto F(x,y)=(p(x,y), q(x,y)),
\end{align*} 
where $p$ and $q$ are immersions in one of the semi-Riemannian factors $\Sl$. These immersions will play a fundamental role in the classification of the degenerate almost complex surfaces. Using the natural product structure $Q$, defined in Equation (\ref{def:Q}), it is possible to study the individual immersions $p$ and $q$, as for any vector field $Z$ on the immersion $F$ one has
\begin{align*}
dp(Z)&=\frac{1}{2}(d F(Z)-QdF(Z)),\\
dq(Z)&=\frac{1}{2}(d F(Z)+Qd F(Z)).
\end{align*} 
One can now consider the immersions $p$ and $q$ with respect to the frame $\left\{X, JX \right\}$ on the immersion $F$. From the definition of $Q$ one has
\begin{align*}
QX&=-\frac{1}{\sqrt{3}}(2PJX-JX)=\sqrt{3}JX,\quad{}Q(JX)=\frac{1}{\sqrt{3}}X,
\end{align*}
as $X$ and $JX$ are eigenvectors of the almost product structure $P$, with eigenvalues $1$ and $-1$, respectively. The immersions $p$ and $q$ with respect to the frame $\left\{X, JX \right\}$ are thus given by
\begin{align*}
d p(X)&=\frac{1}{2}(X-\sqrt{3}JX),\quad{}dp(JX)=-\frac{1}{2\sqrt{3}}(X-\sqrt{3}JX),\\
d q(X)&=\frac{1}{2}(X+\sqrt{3}JX), \quad{}dq(JX)=\frac{1}{2\sqrt{3}}(X+\sqrt{3}JX).
\end{align*}
One can immediately see that $dp(X)$ and $dp(JX)$ are linearly dependent, while the same also holds for $dq(X)$ and $dq(JX)$. Thus the immersions $p$ and $q$ can be seen as curves $a, b$ in $\Sl$, with 
\begin{align*}
a(s)&=\begin{pmatrix}
a_1(s) & a_2(s) \\
a_3(s) & a_4(s)\\
\end{pmatrix},\quad{}
b(t)=\begin{pmatrix}
b_1(t) & b_2(t) \\
b_3(t) & b_4(t)\\
\end{pmatrix}.
\end{align*}
We thus have the immersion $F:\Sigma\rightarrow \SL: (s, t)\mapsto (a(s), b(t))$. The tangent space at each point is then spanned by
\begin{align*}
F_s=(a_s, 0)=(a(s) \alpha(s), 0)\;\; F_t=(0, b_t)=(0, b(t) \beta(t)), 
\end{align*}
with $\alpha, \beta\in\mathfrak{sl}_2\mathbb{R}$. As the immersion $F$ has to be almost complex, one has that the almost complex structure $J$, defined in equation \ref{def:J}, must preserve the tangent bundle of the surface. This means that the vector fields $JF_s$ and $JF_t$ must be linear combinations of the tangent vector fields $F_s$ and $F_t$, where \begin{align*}
JF_s=\frac{1}{\sqrt{3}}(a(s) \alpha(s), 2b(t) \alpha(s))\;\; JF_t=-\frac{1}{\sqrt{3}}(2a(s) \beta(t), b(t) \beta(t)).
\end{align*}
These relations can then be written as 
\begin{align*}
JF_s&=c_1 F_s+c_2 F_t, \quad{} JF_t=c_3 F_s+c_4 F_t, 
\end{align*}
with expressions 
\begin{align*}
c_1&=\frac{1}{\sqrt{3}},\quad{} c_2=\lambda(s,t),\
\quad{}c_3=\tilde{\lambda}(s,t),\quad{} c_4=-\frac{1}{\sqrt{3}}, 
\end{align*}
where $\lambda, \tilde{\lambda}$ are functions on the immersion $F$.
As the almost complex structure $J$ satisfies $J^2=-Id$, one has that
\begin{align*}
J^2F_s&=\frac{1}{\sqrt{3}}\left(\frac{1}{\sqrt{3}} F_s+\lambda F_t\right)+\lambda\left(\bar{\lambda}F_s-\frac{1}{\sqrt{3}}F_t \right),
\end{align*}
which means that $\lambda\bar{\lambda}=-\frac{4}{3}$. The matrices $\alpha$ and $\beta$ then satisfy the equation $\frac{2}{3}\alpha(s)=\lambda(s, t) \beta(t)$. Differentiating this equation, with respect to the coordinate $s$, yields
$\frac{2}{3}\alpha'(s)=\lambda_s(s, t) \beta(t)$, which shows that $\alpha'(s)$ is parallel with respect to $\alpha(s)$. One can then write the curve $\alpha$ as $\alpha(s)=\lambda(s, 0)\beta(0)$. In an analogous way one has for the curve $\beta$ that $\beta(t)=\bar{\lambda}(0, t)\alpha(0)$. The expressions for the matrices $\alpha$ and $\beta$ thus simplify as
\begin{align*}
\alpha(s)&=\lambda_1(s)\beta(0),\;\; \beta(t)=\lambda_2(t)\alpha(0),
\end{align*}
with $\lambda_1$ and $\lambda_2$ functions on the immersion $F$, which can be assumed to  never  be zero. We will now find a re-parametrization to simplify the functions $\lambda_1$ and $\lambda_2$. Take a function $f$ such that 
$\alpha(f(s))=\tilde{\alpha}(\tilde{s})$. This implies that $\tilde{\lambda}_1(\tilde{s})=\lambda_1(f(s))f'(s)$. Thus one can always find a function $f$ such that $\tilde{\lambda}_1(\tilde{s})$ is everywhere equal to 1. This implies that it is possible to take $\alpha(s)=\beta(0)$, i.e. a constant matrix. Analogously it is possible to take $\beta(t)=\alpha(0)$, thus the matrices $\alpha, \beta$ are completely determined by the value of $\alpha(0)$.  The matrix $\alpha(0)$ lies in the Lie algebra $\mathfrak{sl}_2\mathbb{R})$, thus it can be decomposed with respect to the semi-orthonormal basis $\left\{\mathfrak{i}, \mathfrak{j}, \mathfrak{k} \right\}$ as
\begin{align*}
\alpha(0)&=\alpha_1\mathfrak{i}
+\alpha_2\mathfrak{j}+\alpha_3\mathfrak{k},
\end{align*}
with $\alpha_1, \alpha_2, \alpha_3\in\mathbb{R}$. 

Recall that the surface $\Sigma$ is a degenerate surface, i.e. $g(F_s, F_s)=g(F_s, F_t)=g(F_t, F_t)=0$, which can be translated to the condition that the matrix $\alpha(0)$ has to be degenerate, which means $a_1^2+a_2^2-a_3^2=0$. Lemma \ref{lem:isom} shows that the map $\mathcal{F}_{ABC}$ is an isometry of the nearly K\"ahler $\SL$, which also preserves $P$ and $J$, allowing us to fix the constant matrix $\alpha(0)$ by choosing the values $\alpha_1=0, \alpha_2=1, \alpha_3=1$, yielding
\begin{align*}
\alpha(0)&=\begin{pmatrix}
0 & 2 \\
0 & 0 \\
\end{pmatrix}.
\end{align*}
It is then finally possible to obtain an expression for the curve $a(s)$ by analysing the differential equation $a_s(s)=a(s)\alpha(0)$: 
\begin{align*}
\begin{pmatrix}
a'_1(s) & a'_2(s) \\
a'_3(s) & a'_4(s)\\
\end{pmatrix}&=\begin{pmatrix}
a_1(s) & a_2(s) \\
a_3(s) & a_4(s)\\
\end{pmatrix}\cdot\begin{pmatrix}
0 & 2 \\
0 & 0 \\
\end{pmatrix}=\begin{pmatrix}
0 & 2a_1(s) \\
0 & 2a_3(s)\\
\end{pmatrix},
\end{align*}
which has as solution
\begin{align*}
a(s)=\begin{pmatrix}
A & 2As+B \\
C & 2Cs+D\\
\end{pmatrix},
\end{align*}
with $A, B, C, D\in \mathbb{R}$. After once again applying an isometry of the form $\mathcal{F}_{ABC}$ one can fix the initial condition $\alpha(0)=Id$. After applying an analogous method for the curve $b(t)$, the following solutions are obtained:
\begin{align*}
a(s)=\begin{pmatrix}
1 & 2s \\
0 & 1\\
\end{pmatrix},\quad{}b(t)=\begin{pmatrix}
1 & 2t \\
0 & 1\\
\end{pmatrix}.
\end{align*}
The immersed surface $\Sigma$ is thus congruent to an open part of the immersion \\${\mathcal{F}:\mathbb{R}^2\rightarrow \SL}$, given by
\begin{align*}
\mathcal{F}(s, t)&=(\alpha(s), \beta(t))=\left(\begin{pmatrix}
1 & 2s \\
0 & 1\\
\end{pmatrix},\begin{pmatrix}
1 & 2t \\
0 & 1\\
\end{pmatrix}\right).
\end{align*}
\end{proof}

\section{Proof of \ref{Main theorem 2}}\label{section:proofMT2}
In this section we consider the case when the distribution $\mathcal{D}$, as defined in Equation (\ref{def: distrD}), on an immersed, degenerate almost complex surface $\Sigma$ is four-dimensional. A classical result demonstrates that this distribution cannot be degenerate, as in a six-dimensional manifold with a metric of signature two, a two-dimensional subspace represents the largest possible degenerate subspace \cite{oneill}.
\begin{remark}\label{remark:VFX}
	It is then always possible to locally define a unique (up to sign) vector field $X$ on the surface $\Sigma$ such that $\left\{X, JX, PX, JPX\right\}$ is a basis for the distribution $\mathcal{D}$ and that $X$ satisfies 
	\begin{align}
	g(X, PX)=1,\quad{} g(X, JPX)=0.
	\end{align}
Indeed, one can always take the vector field $X$ to be orthogonal to the vector field $JPX$. It is also clear that it is possible to rescale the vector field $X$ such that $g(X, PX) = \pm 1$. Alternatively, one could define a different frame $\left\{Y, JY\right\}$ on $\Sigma$ by letting $Y = JX$. In this case, it becomes clear that $g(Y, PY) = -g(X, PX)$, which means that we can take $g(X, PX)$ to be equal to $1$, without any loss of generality.
\end{remark}
Let us thus consider a local frame $\left\{X, JX\right\}$ on the surface $\Sigma$, as specified in the previous remark. Proposition \ref{prop: G} now immediately shows that the set of vector fields
\begin{align}\label{def:distrF}
\mathcal{G}&=\left\{X, JX, PX, JPX, G(X, PX), JG(X, PX) \right\},
\end{align}
forms a local frame for $\SL$, when restricted to the surface $\Sigma$. The nearly K\"ahler metric $g$ with respect to this frame, using Lemma (\ref{lem:lengthG}), is then given by
\begin{align*}
\begin{pmatrix}
0&0&1&0&0&0\\
0&0&0&1&0&0\\
1&0&0&0&0&0\\
0&1&0&0&0&0\\
0&0&0&0&\frac{2}{3}&0\\
0&0&0&0&0&\frac{2}{3}
\end{pmatrix}.
\end{align*}
\ref{Main theorem 2} shows that every degenerate almost complex surface in $\SL$ for which the almost product structure $P$ does not leave the tangent bundle invariant, is locally determined by an constant angle function $\phi$. We will now prove this theorem.
\begin{proof}[Proof of Main Theorem 2]
Let $\left\{X, JX\right\}$ be the unique local frame as defined in Remark \ref{remark:VFX} and consider the frame $\mathcal{G}$, given in Equation (\ref{def:distrF}). We will begin the proof of \ref{Main theorem 2} by considering how the nearly K\"ahler connection behaves with respect to the frame $\mathcal{G}$. This is accomplished by first defining the vector fields $\nabla_X X$ and $\nabla_{JX}X$ as 
\begin{align*}
\nabla_X X&=a_1 X+ a_2 JX + a_3 PX +a_4 JPX + a_5 G(X, PX)+ a_6 JG(X, PX),\\\nabla_{JX}X&=b_1 X+ b_2 JX + b_3 PX +b_4 JPX + b_5 G(X, PX)+ b_6 JG(X, PX), 
\end{align*}
with $a_1, \ldots, a_6, b_1\ldots, b_6$ functions on the surface $\Sigma$. We will now find find conditions on these functions. Using the properties of the tensor $G$ and $\nabla G$, shown in Lemma \ref{lem:lengthG}, and the almost product structure $P$ and $\nabla P$, given in Lemma \ref{lem:propP}, one can then construct the nearly K\"ahler connection $\nabla$ with respect to the frame $\mathcal{G}$, when restricted to the surface $\Sigma$.
\begin{align*}
&\nabla_X X=a_1 X+ a_2 JX + a_3 PX +a_4 JPX + a_5 G(X, PX)+ a_6 JG(X, PX),\\
&\nabla_X JX=-a_2 X+ a_1 JX - a_4 PX +a_3 JPX - a_6 G(X, PX)+ a_5 JG(X, PX) ,\\
&\nabla_X PX=a_3 X- a_4 JX + a_1 PX -a_2 JPX + a_5 G(X, PX)+ (\frac{1}{2}-a_6) JG(X, PX),\\
&\nabla_X JPX=a_4 X+ a_3 JX + a_2 PX +a_1 JPX + (\frac{1}{2}+a_6) G(X, PX)+ a_5 JG(X, PX),\\
&\nabla_X G(X, PX)=-\frac{2a_5\lambda}{3} X-\frac{1}{3}(1+2a_6)\lambda JX  -\frac{2a_5\lambda}{3} PX +\frac{2a_6\lambda}{3} JPX + 2a_1 G(X, PX),\\
&\nabla_X JG(X, PX)=\frac{1}{3}(-1+2a_6) X-\frac{2a_5\lambda}{3} JX  -\frac{2a_6\lambda}{3} PX -\frac{2a_5\lambda}{3} JPX + 2a_1 JG(X, PX),\\
&\nabla_{JX}X=b_1 X+ b_2 JX + b_3 PX +b_4 JPX + b_5 G(X, PX)+ b_6 JG(X, PX),\\
&\nabla_{JX}JX=-b_2 X+ b_1 JX + -b_4 PX +b_3 JPX  -b_6 G(X, PX)+ b_5 JG(X, PX),\\
&\nabla_{JX}PX=b_3 X+ -b_4 JX + b_1 PX -b_2 JPX + (\frac{1}{2}+b_5) G(X, PX)- b_6 JG(X, PX),\\
&\nabla_{JX}JPX=b_4 X+ b_3 JX + b_2 PX +b_1 JPX + b_6 G(X, PX)+ (-\frac{1}{2}+b_5) JG(X, PX),\\
&\nabla_{JX}G(X, PX)=-\frac{1}{3}(1+ 2b_5)\lambda X -\frac{2b_6\lambda}{3} JX  -\frac{2b_5\lambda}{3} PX +\frac{2b_6\lambda}{3} JPX + 2b_1 G(X, PX),\\
&\nabla_{JX}JG(X, PX)=\frac{2b_6\lambda}{3} X -\frac{1}{3}(1-2b_5) JX  -\frac{2b_6\lambda}{3} PX -\frac{2b_5\lambda}{3} JPX + 2b_1 JG(X, PX).
\end{align*}
As the Levi-Civita connection $\nabla$ of the nearly K\"ahler metric $g$ is compatible with the metric $g$, it implies the following conditions:
\begin{align*}
a_3=0,\quad{} b_3=0,\quad{} a_1=0,\quad{} b_1=0.
\end{align*}
Note, that even though the surface $\Sigma$ is degenerate, the Lie bracket $[X, JX]$ should still be tangent to the surface. Thus $[X, JX]=\nabla_X JX-\nabla_{JX}X$ should only have components in the direction of $X$ and $JX$, yielding
\begin{align*}
b_3=-a_4,\quad{} b_4=a_3,\quad{} b_5=-a_6,\quad{} b_6=a_5.
\end{align*}
Note that, as the surface $\Sigma$ is degenerate, it is not possible to use the  Gauss equation (\ref{Gaussequation}), Codazzi equation (\ref{Codazzi}) and Ricci equation (\ref{Ricci}). Fortunately, it is still useful to look at the expression for the curvature tensor $R$ of the nearly K\"ahler $\SL$, given in Lemma (\ref{lem:R}), as we will compare it with the general definition of the Riemann curvature tensor. In other words, the following equality should always hold:
\begin{align}\label{eq:compareR}
R(U, V, W)=\nabla_U\nabla_V W-\nabla_V\nabla_U W-\nabla_{[U, V]}W,
\end{align} 
for all vector fields $U, V, W$ on $\SL$, where the left hand side is equal to expression shown in Lemma \ref{lem:R}. By substituting $(U, V, W)$ with $(X, JX, X)$ in the previous equation, one obtains the following (differential) equations:
\begin{align}
a_5^2+a_6^2&=1,\label{eq:cossin}\\
E_2(a_5)+E_1(a_6)&=2a_2a_5-2a_6b_2\label{eq:coneq1},\\
E_2(a_6)-E_1(a_5)&=2a_2a_6+2a_5b_2,\label{eq:coneq2}\\
E_2(a_2)-E_1(b_2)&=a_2^2+b_2^2.
\end{align}
Equation (\ref{eq:cossin}) then immediately implies that it is possible to define a smooth function $\phi$ on $\Sigma$ such that
\begin{align*}
a_5&=\cos\phi,\quad{}a_6=\sin\phi.
\end{align*}
This reduces Equations (\ref{eq:coneq1})-(\ref{eq:coneq2}) to 
\begin{align*}
X(\phi)&=2a_2,\quad{}JX(\phi)=2b_2.
\end{align*}
Substituting $(U, V, W)$ in Equation (\ref{eq:compareR}) with $(X, JX, G(X, PX))$ then yields the following equations: 
\begin{align*}
a_2\cos\phi-b_2\sin\phi&=0,\quad{}b_2\cos+a_2\sin\phi=0,
\end{align*}
implying that $a_2=0$ and $b_2=0$. Note that this shows that $X(\phi)=JX(\phi)=0$, making $\phi$ a constant function. The nearly K\"ahler connection with respect to the frame $\mathcal{G}$ is then given by 
\begin{align*}
&\nabla_X X=\cos\phi G(X, PX)+\sin\phi JG(X, PX),\\
&\nabla_X JX=-\sin\phi G(X, PX)+\cos\phi JG(X, PX),\\
&\nabla_X PX=\cos\phi G(X, PX)+\left(\frac{1}{2}-\sin\phi \right)JG(X, PX),\\
&\nabla_X JPX=\left(\frac{1}{2}+\sin\phi \right)G(X, PX)+\cos\phi JG(X, PX),\\
&\nabla_X G(X, PX)=-\frac{2\cos\phi}{3}X+\frac{2}{3}\left(\sin\phi-\frac{2}{3} \right)JX -\frac{2\cos\phi}{3}PX+\frac{2\sin\phi}{3}JPX,\\
&\nabla_X JG(X, PX)=\frac{2}{3}\left(\sin\phi-\frac{1}{2} \right)X-\frac{2\cos\phi}{3}JX-\frac{2\sin\phi}{3}PX-\frac{2\cos\phi}{3}JPX,\\
&\nabla_{JX}X=-\sin\phi G(X, PX)+ \cos\phi JG(X, PX),\\
&\nabla_{JX}JX=-\cos\phi G(X, PX)-\sin\phi JG(X,PX),\\
&\nabla_{JX}PX=\left(\frac{1}{2}-\sin\phi\right) G(X, PX)-\cos\phi JG(X, PX),\\
&\nabla_{JX}JPX=\cos\phi G(X, PX)-\left(\frac{1}{2}+\sin\phi \right)JG(X, PX),\\
&\nabla_{JX}G(X, PX)=-\frac{2}{3}\left(\frac{1}{2}-\sin\phi \right)X-\frac{2\cos\phi}{3}JX+\frac{2\sin\phi}{3}PX+\frac{2\cos\phi}{3}JPX,\\
&\nabla_{JX}JG(X, PX)=\frac{2\cos\phi}{3}X+\frac{2}{3}\left(\frac{1}{2}-\sin\phi \right)JX-\frac{2\cos\phi}{3}PX+\frac{2\sin\phi}{3}JPX.
\end{align*}
Thus $\phi$ is a constant angle function on the surface $\Sigma$ that satisfies 
\begin{align*}
	g(X, G(X, PX))=\frac{3}{2}\cos\phi,
\end{align*} 
thereby proving the first item of \ref{Main theorem 2}.

Suppose now that $F, G:\Sigma\rightarrow\SL$ are two degenerate, almost complex immersions of the surface $\Sigma$ in the nearly K\"ahler $\SL$, and let $\left\{X, Y\right\}$ and $\left\{\overline{X}, \overline{JX}\right\}$ be their respective frames, as defined in Remark \ref{remark:VFX}. It is always possible, after applying an isometry of the ambient space, to take a point $p\in\Sigma$ such that $F(p)=G(p)$ and that $X_p=\overline{X}_p$ and $JX_p=\overline{JX}_p$. If the local angle function $\phi$ and $\overline{\phi}$ of the immersions $F$ and $G$, respectively, are identical ,then induces the same constant angle function $\phi$ on both immersions, then one can see that
\begin{align*}
g(\nabla_U V, W)=g(\nabla_{\overline{U}}\overline{V}, \overline{W}),
\end{align*}
with $U, V,W$ vector fields on the immersion $F$ and $\overline{U},\overline{V},\overline{W}$ vector fields on the immersion $G$. A classical results now shows that the immersions $F$ and $G$ coincide on a local neighbourhood of the point $p$ \cite{griffiths}, proving the second item of \ref{Main theorem 2}.
\end{proof}
\begin{corollary}
An immediate consequence of the proof of \ref{Main theorem 2}, is that the vector fields $X$ and $JX$, as defined in Remark \ref{remark:VFX}, are coordinate vector fields, as the Lie bracket $[X, JX]$ vanishes. Note that the nearly K\"ahler metric $g$ also induces a semi-Riemannian metric $\tilde{g}$ on the surface $\Sigma$ as $\tilde{g}(U, V)=g(U, PV)$, for all vector fields $U$ and $V$ on $\Sigma$. As one straightforwardly has that $\tilde{g}(X, X)=1,\tilde{g}(X, JX)=0$ and $\tilde{g}(JX, JX)=-1$, with $X$ and $JX$ coordinate vector fields, it is possible to conclude that $(\Sigma, \tilde{g})$ has to be locally isometric to the semi-Euclidean plane $\mathbb{R}^2_1$.
\end{corollary}
%
\section{Isoparametric surfaces in $\Sl$}\label{section:Isoparam}
As in Section \ref{section:proofMT1}, when considering an immersion $F:\Sigma\rightarrow\SL:(x, y)\mapsto (p(x, y), q(x, y))$, one can analyze the immersions $p, q:\Sigma\rightarrow\Sl$ in the respective factors $\Sl$, endowed with the usual semi-Riemannian metric $\left\langle\, , \,\right\rangle$. The following section shows that these immersions are flat, isoparametric surfaces in $\SL$, while we also construct a generic null frame on these immersions. We conclude the section by classifying the immersions $p$ and $q$, using the constant angle function $\phi$ defined in \ref{Main theorem 2}.
\subsection{Analyzing the degenerate surface in the factors $\Sl$}\label{section: proppenq}
The following proposition shows that the immersions $p$ and $q$, defined as $F:\Sigma\rightarrow\SL:(x, y)\mapsto (p(x, y), q(x, y))$, are Lorentzian flat isoparametric surfaces in the semi-Riemannian $(\Sl, \left\langle\; ,\;\right\rangle)$, which are then also automatically constant-mean-curvature surfaces. 
\begin{proposition}\label{prop:Isoparam}
Let $F:\Sigma\rightarrow\SL:(x, y)\mapsto(p(x, y), q(x, y))$ be an immersion of a degenerate almost complex surface, for which the almost product structure $P$ does not leave the tangent bundle invariant. Then the immersions $p, q\rightarrow\Sl$ are flat Lorentzian isoparametric surfaces in $(\Sl, \left\langle\, ,\,\right\rangle)$, automatically making them constant-mean-curvature surfaces.
\end{proposition}
\begin{proof}
The first part of the proof will focus on the immersion $p:\Sigma\rightarrow\Sl$, which we will once again analyze using the natural product structure $Q$, defined in Equation (\ref{def:Q}), as shown in the proof of \ref{Main theorem 1}. Recall that for every tangent vector field $Z$ on $\Sigma$ one has that $dp(Z)=\frac{1}{2}(dF(Z)-QdF(Z))$. With respect to the frame $\left\{X, JX \right\}$, defined in the proof of \ref{Main theorem 2}, one then obtains
\begin{align*}
dp(X)&=\frac{1}{2}X-\frac{1}{2\sqrt{3}}JX-\frac{1}{\sqrt{3}}JPX,\quad{}
dp(JX)=\frac{1}{2\sqrt{3}}X+\frac{1}{2}JX-\frac{1}{\sqrt{3}}PX.
\end{align*}
Note that $dp(X)$ and $dp(JX)$ are clearly linearly independent, thus forming a frame for the immersion $p$. The relation between the semi-Euclidean metric $\left\langle\, ,\,\right\rangle $ and the nearly K\"ahler metric $g$, Equation (\ref{def:Euclmetric}), then immediately shows that
\begin{align*}
\left\langle dp(X), dp(X)\right\rangle=\frac{1}{2}, \quad{}\left\langle dp(JX), dp(JX)\right\rangle=-\frac{1}{2},\quad{}\left\langle dp(X), dp(JX)\right\rangle=-\frac{\sqrt{3}}{2},
\end{align*}
making the immersion $p:\Sigma\rightarrow\Sl$ a Lorentzian surface. As $X$ and $ JX$ are coordinate vector fields on the immersion $F:\Sigma\rightarrow\SL$, as shown in the proof of \ref{Main theorem 2}, $dp(X)$ and $dp(JX)$ will also be coordinate vector fields for the immersion $p$. In order to prove that the surface is isoparametric, one should analyze the eigenvalues of the shape-operator with respect to a unit normal vector field. A straightforward calculation shows that the vector field $\eta$, defined as
\begin{align}\label{def:normaleta}
\eta&=-\frac{1}{2}G(X, PX)+\frac{\sqrt{3}}{2}JG(X, PX), 
\end{align}  
is a spacelike unit normal vector field on the immersion $p$ in $\Sl$. It is then sufficient to consider the semi-Euclidean connection $\nabla^E$ on $\Sl$ with respect to the normal vector field $\eta$, thus calculating $\nabla^E_{X}\eta$ and $\nabla^E_{JX}\eta$. Using the frame $\mathcal{G}$ and the connection constructed in the proof of \ref{Main theorem 2}, together with the relation between the semi-Euclidean connection $\nabla^E$ and the nearly K\"ahler connection $\nabla$, specified in Lemma \ref{lem:Euclcon}, yields
\begin{align*}
\nabla^E_{X}\eta&=\left(\cos\phi+\frac{-1+2\sin\phi}{2\sqrt{3}}\right)dp(X)+\left(\frac{1}{2}-\frac{\cos\phi}{\sqrt{3}} +\sin\phi\right)dp(JX),\\
\nabla^E_{JX}\eta&=\left(\frac{1}{2}+\frac{\cos\phi}{\sqrt{3}} -\sin\phi\right)dp(X)+\left(\cos\phi+\frac{1+2\sin\phi}{2\sqrt{3}}\right)dp(JX),
\end{align*}
where $\phi$ is the constant angle function locally determining the immersion $F:\Sigma\rightarrow \SL$. Notice that these expressions simplify, when considering another angle $\psi=\phi+\frac{\pi}{3}$, as
\begin{align}
\nabla^E_{X}\eta&=\frac{-1+4\sin\psi}{2\sqrt{3}}dp(X)+\frac{\sqrt{3}-4\cos\psi}{2\sqrt{3}}dp(JX)\label{eq:ShapeP1},\\
\nabla^E_{JX}\eta&=\frac{\sqrt{3}+4\cos\psi}{2\sqrt{3}}dp(X)+\frac{1+4\sin\psi}{2\sqrt{3}}dp(JX)\label{eq:ShapeP2}.
\end{align}
The eigenvalues of the shape operator of $p$ are then given in function of the constant angle function $\psi$ as 
\begin{align}\label{eq:eigenvP}
\lambda_1&=\frac{1}{\sqrt{3}}\left(-\sqrt{-1-2\cos2\psi}+2\sin\psi \right),\quad{}\lambda_2=\frac{1}{\sqrt{3}}\left(\sqrt{-1-2\cos2\psi}+2\sin\psi \right).
\end{align}
The Gauss equation (\ref{Gaussequation}) for principal curvatures then shows
\begin{align*}
K=c+\lambda_1\lambda_2=0, 
\end{align*}
where $c=-1$ is the sectional curvature of $\Sl$, as it is locally isometric to the anti-de Sitter space $H^3_1$. We can thus conclude that the immersion $p$ a flat surface. As the principal curvatures $\lambda_1, \lambda_2$ are constant,  the immersion $p$ is an isoparametric surface in $\Sl$. The mean curvature vector field $H$ is the trace of the shape operator, immediately showing that
\begin{align*}
H&=\frac{2\sin\psi}{\sqrt{3}}\;\eta,
\end{align*}
making the immersion a constant-mean-curvature surface. We will now follow an analogous approach to analyze the immersion $q:\Sigma\rightarrow\Sl$. With respect to the frame $\left\{X, JX \right\}$ one obtains
\begin{align*}
dq(X)&=\frac{1}{2}X+\frac{1}{2\sqrt{3}}JX+\frac{1}{\sqrt{3}}JPX,\quad{}
dq(JX)=-\frac{1}{2\sqrt{3}}X+\frac{1}{2}JX+\frac{1}{\sqrt{3}}PX, 
\end{align*}
which form a frame for the immersion $q$, with lengths
\begin{align*}
\left\langle dq(X), dq(X)\right\rangle=\frac{1}{2}, \quad{}\left\langle dq(JX), dq(JX)\right\rangle=-\frac{1}{2},\quad{}\left\langle dq(X), dq(JX)\right\rangle=\frac{\sqrt{3}}{2},
\end{align*}
making the immersion $q$ a Lorentzian surface. A straightforward calculation then also shows that the vector field $\omega$, defined as
\begin{align}\label{def:normalomega}
\omega&=\frac{1}{2}G(X, PX)+\frac{\sqrt{3}}{2}JG(X, PX), 
\end{align}  
is a unit normal vector field to the immersion $q$ in $\Sl$. After a substitution of the angle function $\xi=\phi-\frac{\pi}{3}$, the following expressions for the semi-Euclidean connection $\nabla^E$ with respect to the normal vector field $w$ are found: 
\begin{align*}
\nabla_{X}\omega&=\frac{-1+4\sin\xi}{2\sqrt{3}}dq(X)-\frac{\sqrt{3}-4\cos\xi}{2\sqrt{3}}dq(JX),\\
\nabla_{JX}\omega&=\frac{-\sqrt{3}+4\cos\xi}{2\sqrt{3}}dq(X)+\frac{1+4\sin\xi}{2\sqrt{3}}dq(JX).
\end{align*} 
The eigenvalues are then given in function of the constant angle $\xi$ as
\begin{align}\label{eq:eigenvalQ}
\lambda_3&=\frac{1}{\sqrt{3}}\left(-\sqrt{-1-2\cos2\xi}+2\sin\xi \right),\quad{}\lambda_4=\frac{1}{\sqrt{3}}\left(\sqrt{-1-2\cos2\xi}+2\sin\xi \right),
\end{align}
making the immersion $q$ an isoparametric surface, as the principal curvatures are constant. An analogous calculation as done for the immersion $p$ immediately shows that the immersion is a flat surface. The mean curvature vector field $H$ of the immersion has as expression
\begin{align*}
H&=\frac{2\sin\xi}{\sqrt{3}}\omega,
\end{align*} 
which also makes the immersion $q$ a constant mean curvature surface in $\Sl$. 
\end{proof}
We will now construct generic null coordinates on the immersions $p$ and $q$, which will be used in the classification of the degenerate almost complex surfaces $F:\Sigma\rightarrow\SL: (x, y)\rightarrow F(p(x, y), q(x, y))$ in Section \ref{section:types}, while also describing the compatibility conditions between the respective null frames. 
\begin{lemma}\label{lem:null coord}
Let $F:\Sigma\rightarrow\SL:(x, y)\mapsto F((x,y))=(p(x, y), q(x, y))$ be an immersion of a degenerate almost complex surface, for which the almost product structure $P$ does not leave the tangent bundle invariant and consider the immersions $p, q:\Sigma\rightarrow\Sl$. Then one can construct null coordinates $\left\{u,v \right\}$ and $\left\{\tilde{u}, \tilde{v} \right\}$ on the immersions $p$ and $q$, respectively, which satisfy 
\begin{align*}
\tilde{u}&=-\frac{r}{2R}u-\frac{\sqrt{3}}{2rR}v,\quad{}\tilde{v}=\frac{\sqrt{3}rR}{2}u-\frac{R}{2r}v, 
\end{align*}
with $r$ and $R$ non-zero constants on the immersion,  depending on the value of the angle function $\phi$, which locally determines the immersion $F$. 
\end{lemma}
\begin{proof}
From the proof of Proposition \ref{prop:Isoparam} we get that $dp(X)$ and $dp(JX)$ are coordinate vector fields on the immersion $p$, with length 
\begin{align*}
\left\langle dp(X), dp(X)\right\rangle=\frac{1}{2}, \quad{}\left\langle dp(JX), dp(JX)\right\rangle=-\frac{1}{2},\quad{}\left\langle dp(X), dp(JX)\right\rangle=-\frac{\sqrt{3}}{2}.
\end{align*}
Thus one can construct null vector fields $\partial_u, \partial_v$ on the immersion $p$ as 
\begin{align}\label{eq:coorduvp}
\begin{pmatrix}
\partial_u \\
\partial_v
\end{pmatrix}&=\begin{pmatrix}
r\cos\left(\frac{-5\pi}{12}\right) && r\sin\left(\frac{-5\pi}{12} \right)\\
\frac{1}{r}\cos\left(\frac{\pi}{12} \right)&&\frac{1}{r}\sin\left(\frac{\pi}{12} \right)
\end{pmatrix}\begin{pmatrix}
dp(X)\\
dp(JX)
\end{pmatrix},
\end{align}
with lengths 
\begin{align*}
\left\langle\partial_u, \partial_u\right\rangle=\left\langle\partial_v, \partial_v\right\rangle=0,\quad{}\left\langle\partial_u, \partial_v\right\rangle=1,
\end{align*}
and $r$ a non-zero function on the immersion. Analogously one also obtains null vector fields $\partial_{\tilde{u}}, \partial_{\tilde{v}}$ on the immersion $q$, defined as
\begin{align}\label{eq:coorduvq}
\begin{pmatrix}
\partial_{\tilde{u}}\\
\partial_{\tilde{v}}
\end{pmatrix}&=\begin{pmatrix}
R\cos\left(\frac{11\pi}{12} \right) && R\sin\left(\frac{11\pi}{12} \right)\\
\frac{1}{R}\cos\left(\frac{-7\pi}{12} \right) && \frac{1}{R}\sin\left(\frac{-7\pi}{12} \right)
\end{pmatrix}\begin{pmatrix}
dq(X) \\
dq(JX)
\end{pmatrix},
\end{align}
with lengths
\begin{align*}
\left\langle\partial_{\tilde{u}}, \partial_{\tilde{u}}\right\rangle=\left\langle\partial_{\tilde{v}}, \partial_{\tilde{v}}\right\rangle=0,\quad{}\left\langle\partial_{\tilde{u}}, \partial_{\tilde{v}}\right\rangle=1,
\end{align*}
and $R$ a non-zero function on the immersion $q$.
Remark that as $\left\{X, JX \right\}$ were coordinate vector fields on the immersion ${F:\Sigma\rightarrow\SL}$, that $\left\{dp(X), dp(JX) \right\}$ and $\left\{dq(X), dq(JX) \right\}$ denote the same coordinate frame, thus one can write
\begin{align*}
\begin{pmatrix}
\partial_{\tilde{u}}\\
\partial_{\tilde{v}}
\end{pmatrix}&=\begin{pmatrix}
R\cos\left(\frac{11\pi}{12} \right) && R\sin\left(\frac{11\pi}{12} \right)\\
\frac{1}{R}\cos\left(\frac{-7\pi}{12} \right) && \frac{1}{R}\sin\left(\frac{-7\pi}{12} \right)
\end{pmatrix}\begin{pmatrix}
r\cos\left(\frac{-5\pi}{12}\right) && r\sin\left(\frac{-5\pi}{12} \right)\\
\frac{1}{r}\cos\left(\frac{\pi}{12} \right)&&\frac{1}{r}\sin\left(\frac{\pi}{12} \right)
\end{pmatrix}^{-1}\begin{pmatrix}
\partial_{\tilde{u}}\\
\partial{\tilde{v}}
\end{pmatrix}\\
&=-\frac{1}{2}\begin{pmatrix}
\frac{R}{r} && \sqrt{3}rR\\
-\frac{\sqrt{3}}{rR} && \frac{r}{R}\end{pmatrix}\begin{pmatrix}
\partial_{u}\\
\partial_{v}
\end{pmatrix}.
\end{align*}
Applying the chain rule yields
\begin{align*}
\partial_u&=\frac{\partial \tilde{u}}{\partial u}\partial_{\tilde{u}}+\frac{\partial \tilde{v}}{\partial u}\partial_{\tilde{v}},\quad{}\partial_v=\frac{\partial \tilde{u}}{\partial v}\partial_{\tilde{u}}+\frac{\partial \tilde{v}}{\partial v}\partial_{\tilde{v}},
\end{align*}
which results in the following relation between the respective null coordinates:
\begin{align*}
\tilde{u}&=-\frac{r}{2R}u-\frac{\sqrt{3}}{2rR}v,\quad{}\tilde{v}=\frac{\sqrt{3}rR}{2}u-\frac{R}{2r}v.
\end{align*}
We will now look at the derivatives of the immersions $p$ and $q$ with respect to their respective null frames. One thus needs to consider the semi-Euclidean connection $D$ on $R^4_2$, which at every point $A\in \Sl$ behaves as
\begin{align}\label{eq:connH31}
D_X Y=\nabla_X Y+\left\langle X, Y\right\rangle A,
\end{align}
for all vector fields $X, Y$ on $\Sl$. As before, we will first focus on the immersion $p:\Sigma\rightarrow\Sl$. Remark that we will denote $D_{\partial{u}}\partial_u,D_{\partial{u}}\partial_v$ and $D_{\partial{v}}\partial_v$  as $p_{uu}, p_{uv}$ and $p_{vv}$, respectively with respect to the $\left\{u, v \right\}$, as constructed in Lemma \ref{lem:null coord}.  One then get the following expressions for the immersion $p$:
\begin{align*}
p_{uu}=\nabla_{\partial_u}\partial_u,\quad{} p_{uv}=\nabla_{\partial{u}}\partial_v+p,\quad{} p_{vv}=\nabla_{\partial{v}}\partial_v.
\end{align*}
Using the connection coefficients introduced in the proof of \ref{Main theorem 2}, together with the construction of the null frame $\left\{ \partial_u, \partial_v\right\}$ as shown in Equation (\ref{eq:coorduvp}), then yields
\begin{align}\label{eq:connpuv}
p_{uu}&=\frac{r^2(1+2\cos\psi)}{\sqrt{3}}\eta,\quad{}
p_{uv}=\frac{-2}{\sqrt{3}}\sin\psi\,\eta+p,\quad{}
p_{vv}=\frac{1-2\cos\psi}{r^2\sqrt{3}}\eta,
\end{align}
where $\eta$ is the normal vector to the immersion $p$ in $\Sl$. Note that one can always fix the value of the function $r$, thus fixing the coordinates $\left\{u,v\right\}$ by imposing the condition that $\left\langle p_{uu}, p_{uu}\right\rangle=1$, i.e.
\begin{align*}
r^4=\frac{3}{(1+2\cos\psi)^2}.
\end{align*}
Thus $r$ a non-zero constant, depending on the angle $\phi$, which locally determines the surface $\Sigma$ in $\SL$. Note that this expression is well-defined, except in the case where $1+2\cos\psi=0$. In this exceptional case, the value of $r$ can be determined by the requirement that $\left\langle p_{vv}, p_{vv}\right\rangle =1$. Since the immersion $p$ is flat, this condition is equivalent to fixing the value of a component of the second fundamental form

In an analogous way we obtain that the derivatives of the immersion $q$, with respect to the null coordinates $\left\{\tilde{u},\tilde{v} \right\}$, are given by
\begin{align*}
q_{\tilde{u}\tilde{u}}&=\frac{R^2(1+2\cos\xi)}{\sqrt{3}}\omega,\quad{}
q_{\tilde{u}\tilde{v}}=\frac{-2}{\sqrt{3}}\sin\xi\,\omega+q,\quad{}
q_{\tilde{v}\tilde{v}}=\frac{1-2\cos\xi}{R^2\sqrt{3}}\omega,
\end{align*}
where $\omega$ is the normal vector to the immersion $q$, defined in Equation (\ref{def:normalomega}). It is then also possible to fix the coordinates $\left\{\tilde{u},\tilde{v} \right\}$ on $q$ by requiring that $\left\langle q_{\tilde{u}\tilde{u}}, q_{\tilde{u}\tilde{u}}\right\rangle=1$, or $\left\langle q_{\tilde{v}\tilde{v}},q_{\tilde{v}\tilde{v}}\right\rangle=1$ when $1+2\cos\xi=0$. This also makes $R$ a non-zero constant, depending on the value of the constant angle function $\phi$.
\end{proof}

\subsection{Classification of isoparametric Lorentzian surfaces in $\Sl$}\label{section:types}
We have defined the isoparametric immersions $p$ and $q$ as $F:\Sigma\rightarrow\SL$, where $F(x, y)\rightarrow(p(x, y), q(x, y))$. Here, $F$ represents an immersion of a degenerate almost complex surface in which the almost product structure $P$ does not preserve the tangent bundle. \ref{Main theorem 2} shows that this immersion is locally determined by a constant angle function $\phi$. In this section, our objective is to classify the immersions $p$ and $q$ based on this constant angle function.

Recall from Section \ref{section:Correspondence} that there exists a correspondence between the anti-de Sitter space $H^3_1$ and $\Sl$, while Proposition \ref{prop:Isoparam} shows that the immersion $p:\Sigma\rightarrow\Sl$ and $q:\Sigma\rightarrow\Sl$ are Lorentzian isoparametric surfaces. We will now further analyze these immersions using classification results of Lorentzian surfaces in anti-de Sitter space $H^3_1$, found in \cite{xiao}. The following lemma  already implies that one can not always diagonalize the shape operator of a Lorentzian submanifold \cite{book}. 
\begin{lemma}[\cite{book}]
A self-adjoint linear operator $S$ of an $n$-dimensional vector space $V$, endowed with a Lorentzian scalar product $\left\langle\, ,\,\right\rangle$ can be put into one of the following four forms: 
\begin{alignat*}{5}
&I:\begin{pmatrix}
a_1& &0\\
&\ddots &\\
0& & a_n
\end{pmatrix},\;\quad{}&& II:\begin{pmatrix}
a_0 & 0 & & & 0\\
1 & a_0 & & &\\
& & a_3 & &\\
& & & \ddots &\\
0  & & & & a_n
\end{pmatrix},\\
\\
&III:\begin{pmatrix}
a_0 & 0 & 0 & & & 0\\
0 & a_0 & 1 & & &\\
-1& 0 & a_0 & & &\\
&  &  & a_4 & & &\\
& & & & \ddots &\\
0  & & & &  & a_n
\end{pmatrix},\;\quad{}&&IV:\begin{pmatrix}
a_0 & b_0 & & & 0\\
-b_0 & a_0 & & &\\
& & a_3 & &\\
& & & \ddots &\\
0  & & & & a_n
\end{pmatrix},
\end{alignat*}
where $b_0$ is assumed to be non-zero. In cases $I$ and $IV$, $S$ is represented with respect to an orthonormal basis $\left\{v_1, \ldots, v_n \right\}$ satisfying $\left\langle v_1, v_1\right\rangle=-1, \left\langle v_i, v_j\right\rangle=\delta_{ij}, \left\langle v_1, v_i\right\rangle=0$ for $2\leq i, j\leq n,$  while in cases $II$ and $III$ the basis $\left\{v_1, \ldots, v_n \right\}$ is semi-orthonormal satisfying $\left\langle v_1, v_1\right\rangle=\left\langle v_2, v_2\right\rangle=\left\langle v_1, v_i\right\rangle =\left\langle v_2, v_i\right\rangle=0$ for $3\leq i\leq n$, $\left\langle v_1, v_2\right\rangle=-1$ and $\left\langle v_i, v_j\right\rangle=\delta_{ij}$ otherwise.
\end{lemma}
Note that the eigenvalues of types $I, II$ and $III$ are all real, while type $IV$ has $a_0\pm ib_0$ as complex eigenvalues. The next lemma describes how the types of the isoparametric immersions $p$ and $q$ depend on the value of the constant angle $\phi$ locally describing the surface $\Sigma$.
\begin{lemma}\label{lem:Classification of types}
Let $F:\Sigma\rightarrow\SL:(x, y)\mapsto(p(x, y), q(x, y))$ be an immersion of a degenerate almost complex surface, for which the almost product structure $P$ does not leave the tangent bundle invariant. The following table then shows how the types of the Lorentzian isoparamatric immersions $p$ and $q$ depend on the constant angle function $\phi$, locally determining the immersion $F$.
\begin{table}[H]
\centering
\caption{Type classification of the isoparametric immersions $p$ and $q$, with $\psi=\phi+\frac{\pi}{3}$ and $\xi=\phi-\frac{\pi}{3}$.}
\begin{tblr}{ |c||c|c|c|c|c|c|c|c|c|c|c|c|c| } 
\hline
$\phi$ & 0 & - & $\frac{\pi}{3}$ & - & $\frac{2\pi}{3}$ & - & $\pi$ & - & $\frac{4\pi}{3}$ & - & $\frac{5\pi}{3}$& - & $2\pi$\\
$\psi$ & $\frac{\pi}{3}$ & - &$\frac{2\pi}{3}$ & - &$\pi$ & - & $\frac{4\pi}{3}$ & - & $\frac{5\pi}{3}$ & - & 2$\pi$ & - & $\frac{\pi}{3}$\\
$\xi$& $\frac{5\pi}{3}$ & - & 0 & - & $\frac{\pi}{3}$ & - & $\frac{2\pi}{3}$ & - & $\pi$ & - & $\frac{4\pi}{3}$ & - & $\frac{5\pi}{3}$\\
$p$ & $II$ & $I$ & $II$ & $IV$ & $IV$ & $IV$ & $II$& $I$ & $II$ & $IV$ & $IV$ & $IV$ & $II$\\
$q$ & $II$ & $IV$ & $IV$ & $IV$ & $II$ & $I$ & $II$& $IV$ & $IV$ & $IV$ & $II$ & $I$ & $II$\\
\hline
\end{tblr}
\label{table:types}
\end{table}
\end{lemma}
\begin{proof}
\ref{Main theorem 2} states that there exists a constant angle function $\phi$ that locally describes the immersion $F:\Sigma\rightarrow\SL$. Once again we will first consider the immersion $p:\Sigma\rightarrow\Sl$. The proof of Proposition \ref{prop:Isoparam} shows that $p$ is an isoparametric Lorentzian surface and that the eigenvalues of its shape-operator, given by Equation (\ref{eq:eigenvP}), depend on the value of the angle $\psi=\phi+\frac{\pi}{3}$. One can then distinguish three different cases: the eigenvalues are distinct and real if $2\cos2\psi<-1$, the eigenvalues are equal and real if $2\cos2\psi=-1$ and the eigenvalues are complex if $2\cos2\psi>-1$.

The immersion is automatically of type $IV$ if the eigenvalues are complex, thus if $2\cos2\psi>-1$. The shape operator of the immersion with respect to the frame $\left\{dp(X), dp(JX) \right\}$  and normal vector field $\eta$ is given in Equations (\ref{eq:ShapeP1})-(\ref{eq:ShapeP2}) and can be written in matrix form as 
\begin{align*}
A_\eta&=\begin{pmatrix}
\frac{-1+4\sin\psi}{2\sqrt{3}} & \frac{\sqrt{3}-4\cos\psi}{2\sqrt{3}}\\
\frac{\sqrt{3}+4\cos\psi}{2\sqrt{3}}&\frac{1+4\sin\psi}{2\sqrt{3}}
\end{pmatrix}.
\end{align*}
If the eigenvalues are distinct and real, then the shape-operator is trivially diagonalizable, thus the immersion is of type $I$ if $2\cos2\psi<-1$. The eigenvectors of $A_\eta$ are given by
\begin{align*}
v_1&=-\sqrt{3}(1+2\sqrt{-1-2\cos2\phi})dp(X)+(3+4\sqrt{3}\cos\phi)dp(JX),\\v_2&=\sqrt{3}(-1+2\sqrt{-1-2\cos2\phi})dp(X)+(3+4\sqrt{3}\cos\phi)dp(JX).
\end{align*}
The shape-operator is not diagonalizable if the eigenvectors are identical, thus when $2\cos2\psi=-1$. As the immersion is a surface, it then has to be of type $II$. The immersion $p$ is thus of type $I$ if $2\cos2\psi<-1$, of type $II$ if $2\cos2\psi=-1$ and of type $IV$ if $2\cos2\psi>-1$.

Analogously, the immersion $q$ is of type $I$ if $2\cos2\xi<-1$, of type $II$ if $2\cos2\xi=-1$ and of type $IV$ if $2\cos2\xi>-1$, where the angle $\xi$ is defined by $\xi=\phi-\frac{\pi}{3}$ as shown in the proof of Proposition \ref{prop:Isoparam}. Table \ref{table:types} then exactly shows how the types of the immersions $p$ and $q$ depend on the constant angle function $\phi$.
\end{proof}
We will now present classifications for flat isoparametric surfaces in $\Sl$, using the correspondences between the anti-de Sitter space $H^3_1$ and $\Sl$, as shown in Section \ref{section:Correspondence}.
\subsubsection*{\textbf{Isoparametric surfaces of type} $\mathbf{I}$}
The following theorem shows a classification of all flat, isoparametric Lorentzian surfaces of type $I$ in $H^3_1$, with distinct eigenvalues \cite{xiao}. 
\begin{theorem}[\cite{xiao}]\label{theo:typeI}
Let $M$ be a flat, isoparametric Lorentzian surface of type $I$ in $H^3_1$, with distinct eigenvalues. Then $M$ is of type $I$ if and only if $M$ is congruent to an open part of one of the following surfaces:
\begin{itemize}
\item[•]$H^1_1(\sqrt{1-\lambda^2})\times S^1\left(\sqrt{\frac{1-\lambda^2}{\lambda^2}}\right)$, where $-1<\lambda<1$,
\item[•] $S^1_1(\sqrt{\lambda^2-1})\times H^1\left(\sqrt{\frac{\lambda^2-1}{\lambda^2}}\right)$, where $\lambda$ is real and $\lambda^2>1$.
\end{itemize}
\end{theorem} 
The following examples then give explicit parametrizations of these surfaces. 
\begin{example}[Type $I$]\label{example:Ia}
If $M$ is congruent to an op part of the surface $H^1_1(\sqrt{1-\lambda^2})\times S^1\left(\sqrt{\frac{1-\lambda^2}{\lambda^2}}\right)$, with $-1<\lambda<1$, then it is congruent to an open part of the immersion
\begin{align*}
f:\mathbb{R}^2\rightarrow H^3_1:(s,t)\mapsto \left(\frac{1}{\sqrt{1-\lambda^2}}\cos(s),  \frac{1}{\sqrt{1-\lambda^2}}\sin(s), \sqrt{\frac{\lambda^2}{1-\lambda^2}}\cos(t), \sqrt{\frac{\lambda^2}{1-\lambda^2}}\sin(t)\right).
\end{align*}
We will now compute the mean curvature vector field of this immersion, while also constructing a null frame of the surface, which will be used to prove Lemma \ref{lem:TypeI}. The tangent space at every point of the immersion is spanned by the vector fields
\begin{align*}
\partial_s&=\left(-\frac{1}{\sqrt{1-\lambda^2}}\sin(s),  \frac{1}{\sqrt{1-\lambda^2}}\cos(s), 0, 0\right),\quad{}
\partial_t=\left(0,  0, -\sqrt{\frac{\lambda^2}{1-\lambda^2}}\sin(t), \sqrt{\frac{\lambda^2}{1-\lambda^2}}\cos(t)\right).
\end{align*}
The second derivatives of the immersion with respect to the coordinates $\left\{s,t \right\}$ are then given by
\begin{align*}
f_{ss}&=D_{\partial_s}\partial_s=\left(-\frac{1}{\sqrt{1-\lambda^2}}\cos(s),  -\frac{1}{\sqrt{1-\lambda^2}}\sin(s), 0, 0\right),\\
f_{tt}&=D_{\partial_t}\partial_t=\left(0,  0, -\sqrt{\frac{\lambda^2}{1-\lambda^2}}\cos(t), -\sqrt{\frac{\lambda^2}{1-\lambda^2}}\sin(t)\right),\\
f_{st}&=D_{\partial_s}\partial_t=0.
\end{align*}
It is now convenient to introduce an orthonormal frame  $\left\{e_1, e_2\right\}$ on this surface, with  $\left\langle e_1, e_1\right\rangle=-1, \left\langle e_2, e_2\right\rangle =1, \left\langle e_1, e_2\right\rangle=0$, which can be constructed as
\begin{align*}
e_1&=\sqrt{1-\lambda^2}\partial_s,\quad{}
e_2=\sqrt{\frac{1-\lambda^2}{\lambda^2}}\partial_t.
\end{align*}
The second fundamental form $h$, with respect to this orthonormal frame, can be calculated using Equation (\ref{eq:connH31}) and yields
\begin{align*}
h(e_1, e_1)&=D_{e_1}e_1+\left\langle D_{e_1}e_1, e_1\right\rangle e_1-\left\langle D_{e_1}e_1, e_2\right\rangle e_2+\left\langle D_{e_1}e_1, f\right\rangle f,\\
h(e_2, e_1)&=D_{e_2}e_2+\left\langle D_{e_2}e_2, e_1\right\rangle e_1-\left\langle D_{e_2}e_2, e_2\right\rangle e_2+\left\langle D_{e_2}e_2, f\right\rangle f,\\
h(e_1, e_2)&=0,
\end{align*}
where
\begin{align*}
D_{e_1}e_1&=(1-\lambda^2)f_{ss},\quad{}
D_{e_2}e_2=\frac{1-\lambda^2}{\lambda^2}f_{tt},\quad{}
D_{e_1}e_2=0.
\end{align*}From the expressions of the tangent vector fields $\partial_s$ and $\partial_t$, one can then compute the length of the mean curvature vector field $H$, which is given by
\begin{align*}
\left\langle H,H\right\rangle &=\frac{(1+\lambda^2)^2}{4\lambda^2}.
\end{align*}
In the second part of this example, we will construct a null frame $\left\{\partial_u, \partial_v\right\}$ on the surface as
\begin{align*}
\partial_u&=-r e_1+re_2,\quad{} 
\partial_v=\frac{1}{2r}e_1+\frac{1}{2r}e_2,
\end{align*}
where $r$ is a non-zero function on the immersion and $\left\langle \partial_u, \partial_u\right\rangle=0$, $\left\langle \partial_v, \partial_v\right\rangle=0$ and $\left\langle\partial_u, \partial_v\right\rangle=1$. The second fundamental form $h$ with respect to this null frame is now given by
\begin{align*}
h(\partial_u, \partial_u)&=r^2h(e_1, e_1)+r^2h(e_2, e_2),\quad{}
h(\partial_v, \partial_v)=\frac{h(e_1, e_1)+h(e_2, e_2)}{4r^2},\quad{}h(\partial_u, \partial_v)=\frac{-h(e_1, e_1)+h(e_2, e_2)}{2}.
\end{align*}
Notice that it is possible to fix the value of $r$ and thereby fully determine the frame $\left\{\partial_u, \partial_v\right\}$ by normalizing $h(\partial_u, \partial_u)$, that is, $\left\langle h(\partial_u, \partial_u), h(\partial_u, \partial_u)\right\rangle=1$. This condition yields the following requirement:
\begin{align*}
(1-\lambda^2)^2r^4=\lambda^2.
\end{align*}
The null frame $\left\{\partial_u, \partial_v \right\}$ also induces null coordinates $\left\{u,v\right\}$ because
\begin{align*}
\partial_u&=-r\sqrt{1-\lambda^2}\partial_s+r\sqrt{\frac{1-\lambda^2}{\lambda^2}}\partial_t,\quad{}
\partial_v=\frac{1}{2r}\sqrt{1-\lambda^2}\partial_s+\frac{1}{2r}\sqrt{\frac{1-\lambda^2}{\lambda^2}}\partial_t,
\end{align*}
which means that one can write the initial coordinates $\left\{s,t \right\}$ in terms of the null coordinates $\left\{u,v \right\}$ by
\begin{align*}
s&=-r\sqrt{1-\lambda^2}u+\frac{1}{2r}\sqrt{1-\lambda^2}v,\quad{}
t=r\sqrt{\frac{1-\lambda^2}{\lambda^2}}u+\frac{1}{2r}\sqrt{\frac{1-\lambda^2}{\lambda^2}}v.
\end{align*}
We conclude this example by considering the isometry $\mathfrak{f}_1$ between $(H^3_1,\left\langle \, ,\,\right\rangle)$ and $(\Sl, \left\langle\, , \,\right\rangle)$, as shown in Equation (\ref{eq:corr1}). This then yields the corresponding immersion in $\Sl$ as
\begin{align*}
f:\mathbb{R}^2\rightarrow\Sl: (s, t)\mapsto\sqrt{\frac{1}{1-\lambda^2}}\begin{pmatrix}
\cos(s)-\sqrt{\lambda^2}\cos(t) & -\sin(s)+\sqrt{\lambda^2}\sin(t) \\
 \sin(s)+\sqrt{\lambda^2}\sin(t)  & \cos(s)-\sqrt{\lambda^2}\cos(t)
\end{pmatrix}.
\end{align*}
\end{example}
\begin{example}[Type $I$]\label{example:Ib}
If $M$ is congruent to an op part of the surface  $S^1_1(\sqrt{\lambda^2-1})\times H^1\left(\sqrt{\frac{\lambda^2-1}{\lambda^2}}\right)$, where $\lambda$ is real and $\lambda^2>1$, then it is congruent to an open part of the immersion
\begin{align*}
f:\mathbb{R}^2\rightarrow H^3_1:(s,t)\mapsto \left(\frac{1}{\sqrt{\lambda^2-1}}\sinh(s), \sqrt{\frac{\lambda^2}{\lambda^2-1}}\cosh(t),\frac{1}{\sqrt{\lambda^2-1}}\cosh(s),\sqrt{\frac{\lambda^2}{\lambda^2-1}}\sinh(t) \right).
\end{align*}
We will now, as in Example \ref{example:Ia}, compute the mean curvature vector field of this immersion and construct a null frame. As this follows a completely analogous procedure, a couple of details will be omitted in this example. The tangent space at every point of the immersion is spanned by the vector fields
\begin{align*}
\partial_s&=\left(\frac{1}{\sqrt{\lambda^2-1}}\cosh(s), 0,\frac{1}{\sqrt{\lambda^2-1}}\sinh(s),0 \right),\quad{}
\partial_t=\left(0, \sqrt{\frac{\lambda^2}{\lambda^2-1}}\sinh(t),0,\sqrt{\frac{\lambda^2}{\lambda^2-1}}\cosh(t) \right).
\end{align*}
One can then construct an orthonormal frame $\left\{e_1, e_2 \right\}$ with $\left\langle e_1, e_1\right\rangle=-1, \left\langle e_2, e_2\right\rangle=1$ and $\left\langle e_1, e_2\right\rangle =0$ by 
\begin{align*}
e_1&=\sqrt{\lambda^2-1}\partial_s,\quad{} e_2=\sqrt{\frac{\lambda^2-1}{\lambda^2}}\partial_t.
\end{align*}
Through a completely analogous argument as in Example \ref{example:Ia}, we find that the length of the mean curvature vector field is given by
\begin{align*}
\left\langle H,H\right\rangle &=\frac{(1+\lambda^2)^2}{4\lambda^2}.
\end{align*}
Once again, one can construct a null frame $\left\{u,v \right\}$ through
\begin{align*}
\partial_u&=-be_1+be_2\quad{}\partial_v=\frac{1}{2r}e_1+\frac{1}{2r}e_2,
\end{align*}
where $r$ is a non-zero function on the immersion. The value of $r$ can be fixed by normalizing $h(\partial_u, \partial_u)$, which yields the condition
\begin{align*}
(1-\lambda^2)^2r^4=\lambda^2.
\end{align*}
Note that the both the length of the mean curvature vector field as the normalization condition for the function $r$ are exactly the same as in Example \ref{example:Ia}. The original coordinates $\left\{s,t\right\}$ can also be related to the null coordinates $\left\{u,v \right\}$ through
\begin{align*}
s&=-r\sqrt{\lambda^2-1}u+\frac{1}{2r}\sqrt{\lambda^2-1}v,\quad{}
t=r\sqrt{\frac{\lambda^2-1}{\lambda^2}}u+\frac{1}{2r}\sqrt{\frac{\lambda^2-1}{\lambda^2}}v.
\end{align*}
The isometry $\mathfrak{f}_1$ between $(H^3_1,\left\langle\, ,\,\right\rangle)$ and $(\Sl, \left\langle\, , \,\right\rangle)$, as shown in Equation (\ref{eq:corr1}), then yields the corresponding immersion in $\Sl$ as
\begin{align*}
f:\mathbb{R}^2\rightarrow\Sl:(s,t)\mapsto\sqrt{\frac{1}{\lambda^2-1}}\begin{pmatrix}
-e^{-s} & -\sqrt{\lambda^2}e^{-t}\\
\sqrt{\lambda^2}e^t & e^s
\end{pmatrix}.
\end{align*}
\end{example}
These examples then allow for the following lemma, which shows a congruence result for when the immersions $p$ and $q$, defined as $F:\Sigma\rightarrow\SL:(x, y)\mapsto(p(x, y), q(x,y))$, are of type $I$.
\begin{lemma}\label{lem:TypeI}
Let $F:\Sigma\rightarrow\SL:(x, y)\mapsto(p(x, y), q(x, y))$ be an immersion of a degenerate almost complex surface, for which the almost product structure $P$ does not preserve the tangent bundle and let $\phi$ be the constant angle function locally determining $F$. If the immersion $p$ is of type $I$, then it is congruent to an open part of either the immersion given in Example \ref{example:Ia} or the immersion given in Example \ref{example:Ib}, with additional congruency condition
\begin{align*}
3(1+\lambda^2)^2&=\lambda^2\sin^2\psi,
\end{align*}
where $\psi=\phi+\frac{\pi}{3}$. The value of  $\lambda$ determines which immersion is congruent to $p$, as $\lambda^2<1$ yields the immersion in Example \ref{example:Ia}, while $\lambda^2>1$ yields the immersion in Example \ref{example:Ib}. A completely analogous statement holds for the immersion $q$ is of type $I$, if one replaces the constant angle function $\psi$ by the constant angle $\xi=\phi-\frac{\pi}{3}$.
\end{lemma}
\begin{proof}
We will proof the statement for the immersion $p:\Sigma\rightarrow\Sl$, as the proof for the immersion $q$ is completely analogous. \ref{Main theorem 2} shows that the immersion $F$ is locally determined by a constant angle function $\phi$, as it is a degenerate almost complex surface, for which the almost product structure $P$ does not preserve the tangent bundle. Proposition \ref{prop:Isoparam} then shows that the immersions $p$ is a Lorentzian isoparametric surface in $\Sl$, and the length of its mean curvature vector field $H$ is given by 
\begin{align*}
\left\langle H, H\right\rangle=\frac{4\sin^2\psi}{3},
\end{align*}
where the angle $\psi$ is given by $\psi=\phi+\frac{\pi}{3}$. Lemma \ref{lem:Classification of types} shows that if the immersion $p$ is of type $I$, then it must have distinct eigenvalues, while applying Theorem \ref{theo:typeI} yields that Lorentzian isoparametric surfaces with distinct eigenvalues in the anti-de Sitter space $H^3_1$ are congruent to either the surface in Example \ref{example:Ia} or to the surface in Example \ref{example:Ib}. After applying the correspondence between $H^3_1$ and $\Sl$, one can see that the immersion $p$ is congruent to either one of the aforementioned examples, if the lengths of the mean curvature vector fields coincide, thus if
\begin{align}\label{eq:congI}
3(1+\lambda^2)^2&=\lambda^2\sin^2\psi,
\end{align}
which is the same condition for both examples. Recall that the parameter $\lambda$ of the surfaces satisfies $\lambda^2<1$ if it is congruent to Example \ref{example:Ia} and if $\lambda^2<1$ then it is congruent to Example \ref{example:Ib}. Note that Lemma \ref{lem:Classification of types} ensures that Equation (\ref{eq:congI}) never has $\lambda$ equal to 1 as a solution, and that the solution for $\lambda$ will then determine the congruent surface of the immersion $p$, thereby proving the lemma. One immediately obtains the same result for the immersion $q$, when one substitutes the angle $\psi$ by the constant angle function $\xi$, defined as $\xi=\phi-\frac{\pi}{3}$.
\end{proof}
\subsubsection*{\textbf{Isoparametric surfaces of type} $\mathbf{II}$}
The following lemma shows a congruence result for when the immersions $p$ and $q$, defined as $F:\Sigma\rightarrow\SL:F(x, y)\rightarrow(p(x, y), q(x, y))$, are of type $II$. Remark that the classification includes equivalent surfaces, that can be obtained by interchanging the coordinates $u$ and $v$. However, for the sake of clarity and to highlight their significance in the classification of degenerate almost complex surfaces in Section \ref{section:classifyingresult}, we will present them in their original form.
\begin{lemma}\label{lem:TypeII}
Let $F:\Sigma\rightarrow\SL:(x, y)\mapsto(p(x, y), q(x, y))$ be an immersion of a degenerate almost complex surface, for which the almost product structure $P$ does not preserve the tangent bundle and let $\phi$ be the constant angle function locally determining $F$.
If the immersion $p$ is of type $II$, then the value of $\psi=\phi+\frac{\pi}{3}$ lies in the set $\left\{\frac{\pi}{3},\frac{2\pi}{3}, \frac{4\pi}{3}, \frac{5\pi}{3}\right\}$. The immersion is then congruent to an open part of the immersion $f:\mathbb{R}^2\rightarrow\Sl$, where the expression of $f$ depends on the value of $\psi$:
\begin{align*}
&\text{type $II_a$: if }\psi=\frac{\pi}{3} \text{ then } f:\mathbb{R}^2\rightarrow \Sl:(u,v)\mapsto \begin{pmatrix}
\frac{(1-u+2v)}{2}e^u & -\frac{(1+u-2v)}{2}e^{-u} \\
e^u & e^{-u}
\end{pmatrix},\\
&\text{type $II_b$: if }\psi=\frac{2\pi}{3} \text{ then }f:\mathbb{R}^2\rightarrow \Sl:(u,v)\mapsto \begin{pmatrix}
\frac{(1-v+2u)}{2}e^v & -\frac{(1+v-2u)}{2}e^{-v} \\
e^v & e^{-v}
\end{pmatrix},\\
&\text{type $II_c$: if }\psi=\frac{4\pi}{3} \text{ then }f:\mathbb{R}^2\rightarrow \Sl:(u,v)\mapsto \begin{pmatrix}
\frac{-\sin(v)+(v+2u)\cos(v)}{2} & \frac{\cos(v)+(v+2u)\sin(v)}{2}\\
-2\cos(v) & -2\sin(v)
\end{pmatrix},\\
&\text{type $II_d$: if }\psi=\frac{5\pi}{3} \text{ then }f:\mathbb{R}^2\rightarrow \Sl:(u,v)\mapsto\begin{pmatrix}
\frac{-\sin(u)+(u+2v)\cos(u)}{2} & \frac{\cos(u)+(u+2v)\sin(u)}{2}\\
-2\cos(u) & -2\sin(u)
\end{pmatrix}.
\end{align*}
A completely analogous statement holds for the immersion $q$, when one replaces the angle function $\psi$ by $\xi=\phi-\frac{\pi}{3}$.
\end{lemma}
\begin{proof}
We will present the proof of the lemma for the immersion $p$, as the immersion $q$ would be completely analogous. Recall from Lemma \ref{lem:Classification of types} that the immersion $p$ is of type $II$ if the angle $\psi=\phi+\frac{\pi}{3}$ has a value lying in the set $\left\{\frac{\pi}{3},\frac{2\pi}{3},\frac{4\pi}{3},\frac{5\pi}{3} \right\}$. One can then, using Lemma \ref{lem:null coord}, put generic null coordinates $\left\{u,v \right\}$, depending on the value of $\psi$ on the immersion $p$. Equation (\ref{eq:connpuv}) then shows that the derivative $p_{uu}$ is zero when $\psi$ is either $\frac{2\pi}{3}$ or $\frac{4\pi}{3}$, while $p_{vv}=0$ when $\psi$ is either $\frac{\pi}{3}$ or $\frac{5\pi}{3}$. It is then sufficient to only consider the case when $p_{vv}=0$, as performing a change of coordinates $u \leftrightarrow v$ will yield the congruency when $p_{uu}=0$. Equation (\ref{eq:connpuv}) now reduces to 
\begin{align}\label{eq:derivpII}
p_{uu}&=\frac{2r^2}{\sqrt{3}}\eta,\quad{} p_{uv}=\mp\eta+p,\quad{}p_{vv}=0,
\end{align}
where the sign of $\eta$ in the expression for $p_{uv}$ is negative if $\psi=\frac{\pi}{3}$ and positive if $\psi$ is equal to $\frac{5\pi}{3}$. We will then, as in Lemma \ref{lem:null coord}, fix our coordinate frame $\left\{u,v \right\}$ by normalizing a component of the second fundamental form, which can be obtained through the condition $\left\langle p_{uu}, p_{uu}\right\rangle=1$. Note that this is trivially satisfied by taking $r=\frac{\sqrt[4]{3}}{\sqrt{2}}$. Notice that the vanishing of $p_{vv}$ implies that the immersion $p$ is linear in the coordinate $v$, thus one can write $p(u,v)=a(u)+vb(u)$, where $a$ and $b$ are smooth functions on $p$ which only depending on the coordinate $u$. The derivatives of the immersion with respect to the these coordinates are given by 
\begin{alignat*}{5}
&p_u=a'(u)+vb'(u),\;\quad{}&&
p_v=b(u),\\
&p_{uu}=a''(u)+vb''(u),\;\quad{}&&
p_{uv}=b'(u),\;\quad{}&&
p_{vv}=0.
\end{alignat*}
To simplify the calculations in this proof, we will consider the isometry $\mathfrak{f}_2$  between $(H^3_1,\left\langle\, ,\,\right\rangle_*)$ and $(\Sl, \left\langle\, , \,\right\rangle)$, as shown in Equation (\ref{eq:corr2}). Thus we will now assume that the immersion $p$ is lying in the anti-de Sitter space $(H^3_1,\left\langle\, ,\,\right\rangle_*)$. The following relations hold since $\left\{u,v \right\}$ are null coordinates, meaning $\left\langle p_u, p_u\right\rangle_*=0, \left\langle p_v, p_v\right\rangle_*=0$ and $\left\langle p_u, p_v\right\rangle_*=1$, while also using that $\left\langle p, p\right\rangle_*=-1$ as $p$ is an immersion in $\Sl$:
\begin{alignat*}{5}
&\left\langle a(u), a(u) \right\rangle_*=-1,\;\quad{} &&
\left\langle b(u), b(u) \right\rangle_* =0,\;\quad{} &&
\left\langle b'(u), b'(u) \right\rangle_* =0,\\
&\left\langle b(u), b'(u) \right\rangle_* =0,\;\quad{} &&
\left\langle a(u), b(u) \right\rangle_* =0,\;\quad{} &&
\left\langle a(u), b'(u) \right\rangle_* =-1,\\
&\left\langle a'(u), b'(u) \right\rangle_* =0,\;\quad{}&&
\left\langle a'(u), a'(u) \right\rangle_* =0,\;\quad{}&& \left\langle a(u), a'(u) \right\rangle_*=0,\;\quad{}&& \left\langle a'(u), b(u)\right\rangle_*=1.
\end{alignat*}
A first implication of the above relations is that $\left\{a, a', b, b' \right\}$ is a frame on $\mathbb{R}^4_2$. This means that one can write $b''=\alpha_1 a+\alpha_2 a' + \alpha_3 b+ \alpha_4 b'$, which reduces to $b''=\alpha_3 b$, when applying the above relations. The curve $b$ thus lies in the degenerate plane spanned by $b(0)$ and $b'(0)$ and for simplicity we will denote $\alpha_3$ as $\alpha$ from now on. The differential equation for the curve $b$ becomes $b''(u)=\alpha(u)b(u)$, where $\alpha$ can be seen as the centro-affine curvature of the curve $b$. As this is a curve in the degenerate plane, one can show that 
\begin{align*}
|b\;b'|=b_1b'_2-b'_1b_2=c,
\end{align*}
where $c\in\mathbb{R}$ is a constant. The aim is now to find explicit expressions for the curves $a$ and $b$, which can already be written as
\begin{align*}
b(u)&=(b_1(u), b_2(u), 0,0),\quad{} a(u)=(a_1(u), a_2(u), a_3(u), a_4(u)).
\end{align*}
The relations $\left\langle a(u), b(u)\right\rangle_*=0$ and $\left\langle a(u), b'(u)\right\rangle_*=-1$ yield
\begin{align*}
&\left\langle a(u), b(u) \right\rangle_*=b_2a_3+b_1a_4=0,\quad{}
\left\langle a(u), b'(u) \right\rangle_*=\frac{1}{2}\left( b'_2a_3+b'_1a_4\right)=-1.
\end{align*}
Solving this system of equations results in the following expressions:
\begin{align*}
a_3&=-\frac{2b_1}{c},\quad{}
a_4=\frac{2b_2}{c},
\end{align*}
where we used that $|b\;b'|=c$. The condition that $\left\langle a(u), a(u)\right\rangle_* =-1$, yields that $2a_1\frac{b_2}{c}-2a_2\frac{b_1}{c}=-1$, which implies that one can write the following expressions for the functions $a_1$ and $a_2$:
\begin{align*}
a_1&=\frac{1}{2}\left(b'_1+\kappa(u)b_1\right),\quad{} a_2=\frac{1}{2}\left(b'_2+\kappa(u)b_2\right),
\end{align*}
with $\kappa$ a smooth function on the immersion. As it should hold that $\left\langle a'(u), a'(u) \right\rangle_* =0$, with
\begin{alignat*}{5}
&a'_1=\frac{1}{2}\left(\alpha(u)b_1+\kappa'(u)b_1+\kappa(u)b'_1\right),\;\quad{}&& a'_3=-2\frac{b'_1}{c},\\
&a'_2=\frac{1}{2}\left(\alpha(u)b_2+\kappa'(u)b_2+\kappa(u)b'_2\right),\;\quad{}&& a'_4=-2\frac{b'_2}{c},
\end{alignat*}
one obtains that $\kappa'(u)=-\alpha(u)$. We will now determine the possible values for the centro-affine curvature $\alpha$, using the  expression for the derivatives of the immersion as given in Equation (\ref{eq:derivpII}). As the frame was fixed such that $\left\langle p_{uu}, p_{uu}\right\rangle_*=1$, implying that $\left\langle a''(u), a''(u)\right\rangle_*=1$, one obtains after a straightforward calculation that $\alpha^2=1$. The expression $p_{uv}=\mp\eta+p$ can then be rewritten as $p_{uv}=\mp p_{uu}+p$, which yields
\begin{align*}
b'(u)&=\mp a''(u)\mp vb''(u)+a(u)+vb(u)\\
&=a(u)\mp a''(u)+v(b(u)\mp\alpha b(u)),
\end{align*}
with $b'(u)=a(u)-a''(u)$. The previous equation will thus only hold if $\alpha=1$ and the angle $\psi$ is equal to $\frac{\pi}{3}$, or if $\alpha=-1$ and $\psi$ is equal to $\frac{5\pi}{3}$. We will now consider both cases separately.


\subsubsection*{\textbf{Angle} $\mathbf{\psi}$ \textbf{is equal to} $\mathbf{\frac{\pi}{3}}$}\label{example:II1}
In this case the function $\alpha$ is equal to $1$, resulting in the differential equation $b''(u)=b(u)$. The curve $b$, lying in the degenerate plan, can then be written as $b(u)=c_1e^u+c_2e^{-u}$, with $c_1, c_2\in\mathbb{R}^4_2$. By applying isometries of $\mathbb{R}^4_2$, i.e. elements of the group $SO(4, 2)$, one can take the vectors $c_1, c_2$ lying in the degenerate plane as 
\begin{align*}
c_1&=(1,0,0,0), \quad{} c_2=(0,1,0,0).
\end{align*}
The function $\kappa$, defined as $\kappa(u)=-\alpha(u)$, can now explicitly be written as $\kappa(u)=-u+C$, with $C\in\mathbb{R}$ a constant. The curves $a$ and $b$ are then described as
\begin{align*}
b(u)&=(e^u, e^{-u}, 0, 0),\quad{} a(u)=\left(\frac{(1-u+C)}{2}e^u, -\frac{(1+u-C)}{2}e^{-u}, -\frac{2e^u}{c}, \frac{2e^{-u}}{c}\right).
\end{align*}
Note that this also fixes the arc length parameter $c$ of the curve $b$ as $|b\;b'|=b_1b'_2-b'_1b_2=c=-2$. The immersion $p$ is thus congruent to an open part of the immersion
\begin{align*}
f:\mathbb{R}^2\rightarrow H^3_1:(u,v)\mapsto\left(\frac{(1-u+C+2v)}{2}e^u, -\frac{(1+u-C-2v)}{2}e^{-u}, e^u, -e^{-u}\right),
\end{align*}
where the constant $C$ can be chosen to be zero, by a translation in the $v$-coordinate. Applying the isometry $\mathfrak{f}_2$ in Equation (\ref{eq:corr2}) shows that the immersion $p$ in $\Sl$ has to be congruent to an open part of the immersion 
\begin{align*}
f:\mathbb{R}^2\rightarrow \Sl:(u,v)\mapsto \begin{pmatrix}
\frac{(1-u+2v)}{2}e^u & -\frac{(1+u-2v)}{2}e^{-u} \\
e^u & e^{-u}
\end{pmatrix}.
\end{align*} 
\subsubsection*{\textbf{Angle} $\mathbf{\psi}$ \textbf{is equal to} $\mathbf{\frac{5\pi}{3}}$}\label{example:II2}
In this case $\alpha$ is equal to $-1$, implying that the curve $b$ has to satisfy the differential equation $b''(u)=-b(u)$. The curve is thus of the form $b(u)=c_3\cos(u)+c_4\sin(u)$, with $c_3, c_4\in\mathbb{R}^4_2$. One can once again choose the vectors $c_3, c_4$ in the degenerate plane such that 
\begin{align*}
c_3 &=(1, 0, 0, 0)\quad{} c_4=(0, 1, 0, 0).
\end{align*}
This choice of vectors automatically fixes the value of the arc length parameter $c$ as $|b\;b'|=b_1b'_2-b'_1b_2=c=1$. The function $\kappa$ is now also determined as $\kappa(u)=u+C$, with $C\in\mathbb{R}$ a constant. The curves $a$ and $b$ have as expressions
\begin{align*}
b(u)&=\left(\cos(u), \sin(u), 0,0 \right),\\
a(u)&=\left(\frac{-\sin(u)+(u+C)\cos(u)}{2}, \frac{\cos(u)+(u+C)\sin(u)}{2}, -2\cos(u), 2\sin(u)\right),
\end{align*}  
The immersion $p$ in $H^3_1$ is then congruent to an open part of the immersion 
\begin{align*}
f:\mathbb{R}^2\rightarrow H^3_1:(u,v)\mapsto\left(\frac{-\sin(u)+(u+C+2v)\cos(u)}{2}, \frac{\cos(u)+(u+C+2v)\sin(u)}{2}, -2\cos(u), 2\sin(u) \right),
\end{align*}
where we can once again take $C$ to be zero after a suitable translation of the coordinate $v$. The isometry $\mathfrak{f}_2$ between $H^3_1$ and $\Sl$ then finally shows that the immersion $p$ in $\Sl$ has to be congruent to an open part of the immersion 
\begin{align*}
f:\mathbb{R}^2\rightarrow \Sl:(u,v)\mapsto\begin{pmatrix}
\frac{-\sin(u)+(u+2v)\cos(u)}{2} & \frac{\cos(u)+(u+2v)\sin(u)}{2}\\
-2\cos(u) & -2\sin(u)
\end{pmatrix}.
\end{align*}
Note that the proof for the immersion $q$ would be completely analogous, when one replaces the constant angle function $\psi$ with the angle function $\xi=\phi-\frac{\pi}{3}$.
\end{proof}
\subsubsection*{\textbf{Isoparametric surfaces of type} $\mathbf{IV}$}
The following theorem shows a classification of all flat isoparametric Lorentzian surfaces of type $IV$ in $H^3_1$ \cite{xiao}.
\begin{theorem}[\cite{xiao}]
Let $M$ be a flat isoparametric Lorentzian surface of type $IV$ in $H^3_1$. Then $M$ is congruent to an open part of a principal orbit of $G\subset O(4, 2)$ in $H^3_1$, where
\begin{align*}
G=\begin{pmatrix}
\cos(s)\cosh(t) & \cosh(t)\sin(s) & \cos(s)\sinh(t) & \sin(s)\sinh(t)\\
-\sin(s)\cosh(t) & \cos(s)\cosh(t) & -\sin(s)\sinh(t) & \cos(s)\sinh(t)\\
\cos(s)\sinh(t) & \sin(s)\sinh(t) & \cos(s)\cosh(t) & \sin(s)\cosh(t)\\
-\sin(s)\sinh(t) & \cos(s)\sinh(t) & -\sin(s)\cosh(t) & \cos(s)\cosh(t)
\end{pmatrix},
\end{align*}
with $s, t\in\mathbb{R}$.
\end{theorem}
We will now once again explicitly describe a surface of this type, as shown in the following example.
\begin{example}[Type $IV$]\label{example:IV}
As $M$ should be congruent to an open part of a principal orbit of $G$, it has to be congruent to $G\cdot a$, with $a\in H^3_1$. Thus one can consider orbits of the form 
\begin{align*}
\begin{pmatrix}
\cos(s)\cosh(t) & \cosh(t)\sin(s) & \cos(s)\sinh(t) & \sin(s)\sinh(t)\\
-\sin(s)\cosh(t) & \cos(s)\cosh(t) & -\sin(s)\sinh(t) & \cos(s)\sinh(t)\\
\cos(s)\sinh(t) & \sin(s)\sinh(t) & \cos(s)\cosh(t) & \sin(s)\cosh(t)\\
-\sin(s)\sinh(t) & \cos(s)\sinh(t) & -\sin(s)\cosh(t) & \cos(s)\cosh(t)
\end{pmatrix}\cdot \begin{pmatrix}
y_1\\
y_2\\
y_3\\
y_4
\end{pmatrix}, 
\end{align*}
with $y_1, y_2, y_3, y_4\in\mathbb{R}$ satisfying $-y_1^2-y_2^2+y_3^2+y_4^2=-1.$ By applying isometries from the ambient space $H^3_1$, i.e. elements of the group $SO(4, 2)$, one can see that this orbit is congruent to a principal orbit with $y_1=y_4=0$. Then one has that $-y_2^2+y_4^2=-1$, thus there is an $\alpha\in\mathbb{R}$ such that $y_2=\cosh(\alpha)$ and $y_4=\sinh(\alpha)$. Thus $M$ has to be congruent to an open part of the immersion $f:\mathbb{R}^2\rightarrow H^3_1:(s,t)\rightarrow f(s,t)$, with
\begin{align*}
F(s, t)&=\begin{pmatrix}
\cos(s)\cosh(t)\sinh(\alpha)+\cos(s)\sinh(t)\sinh(\alpha)\\
\cos(s)\cosh(t)\cosh(\alpha)-\sin(s)\sinh(t)\sinh(\alpha)\\
\sin(s)\sinh(t)\cosh(\alpha)+\cos(s)\cosh(t)\sinh(a)\\
\cos(s)\sinh(t)\cosh(\alpha)-\sin(s)\cosh(t)\sinh(\alpha)
\end{pmatrix}.
\end{align*}
We will now, as in Example \ref{example:Ia}, determine the length of the mean curvature vector of the immersion, while also constructing a null frame. The tangent space at every point of the immersion is spanned by the vector fields
\begin{align*}
\partial_s&=\begin{pmatrix}
\cos(s)\cosh(t)\cosh(a)-\sin(s)\sinh(t)\sinh(\alpha)\\ -\sin(s)\cosh(t)\cosh(\alpha)-\cos(s)\sinh(t)\sinh(\alpha)\\ \cos(s)\sinh(t)\cosh(\alpha)-\sin(s)\cosh(t)\sinh(\alpha)\\ -\sin(s)\sinh(t)\cosh(\alpha)-\cos(s)\cosh(t)\sinh(\alpha)
\end{pmatrix},\;
\partial_t&=\begin{pmatrix}
\sin(s)\sinh(t)\cosh(\alpha)+\cos(s)\cosh(t)\sinh(\alpha)\\
\cos(s)\sinh(t)\cosh(\alpha)-\sin(s)\cosh(t)\sinh(\alpha)\\
\sin(s)\cosh(t)\cosh(\alpha)+\cos(s)\sinh(t)\sinh(\alpha)\\
\cos(s)\cosh(t)\cosh(\alpha)-\sin(s)\sinh(t)\sinh(\alpha)
\end{pmatrix}
\end{align*}
One can construct an orthonormal frame $\left\{e_1, e_2 \right\}$ with $\left\langle e_1, e_1\right\rangle=-1, \left\langle e_2, e_2\right\rangle=1$ and $\left\langle e_1, e_2\right\rangle=0$ on the immersion by 
\begin{align*}
e1&=\partial_s,\quad{}
e2=\frac{1}{\cosh(2\alpha)}\left(\partial_t-\sinh(2\alpha)\partial_s\right).
\end{align*}
Once again, through an analogous approach as in Example \ref{example:Ia}, one can calculate the length of the mean curvature vector field, given by
\begin{align*}
\left\langle H, H\right\rangle=\tanh^2(2\alpha).
\end{align*} 
It is also possible to construct a null frame $\left\{\partial_u, \partial_v \right\}$ with $\left\langle\partial_u, \partial_u\right\rangle=\left\langle\partial_v, \partial_v\right\rangle=0$ and $\left\langle\partial_u, \partial_v\right\rangle=1$, with
\begin{align*}
\partial_u&=-re_1+re_2,\;\; \partial_v=\frac{1}{2r}e_1+\frac{1}{2r}e_2,
\end{align*}
with $r$ a non-zero function on the immersion. The value of this function can then be fixed by normalizing $h(\partial_u, \partial_u)$, yielding the condition
\begin{align*}
4r^4&=\cosh^2(2\alpha)e^{-4\alpha}.
\end{align*}
One can now construct the coordinates $\left\{u,v\right\}$ through the frame $\left\{\partial_u, \partial_v \right\}$ as
\begin{align*}
\partial_u&=-(r+r\tanh(2\alpha))\partial_s+r\sech(2\alpha)\partial_t,\quad{}
\partial_v=\frac{1}{2r}(1-\tanh(2\alpha)\partial_s+\frac{1}{2r}\sech(2\alpha)\partial_t.
\end{align*}
Thus the coordinates $\left\{s,t \right\}$ can be written in terms of the null coordinates $\left\{u,v \right\}$ as
\begin{align*}
s&=-(r+r\tanh(2\alpha))u+\frac{1}{2r}(1-\tanh(2\alpha))v,\quad{}
t=r\sech(2\alpha)u+\frac{1}{2r}\sech(2\alpha)v.
\end{align*}
We conclude this example by considering the isometry $\mathfrak{f}_1$ between $(H^3_1,\left\langle\, ,\,\right\rangle)$ and $(\Sl, \left\langle\, , \,\right\rangle)$, as shown in Equation (\ref{eq:corr1}), yielding the corresponding immersion in $\Sl$ as
\begin{align*}
f:\mathbb{R}^2\rightarrow\Sl: (s,t)\mapsto\begin{pmatrix}
e^{-t}(\sin(s)\cosh(\alpha)-\cos(s)\sinh(\alpha)) &-e^{-t}(\cos(s)\cosh(\alpha)-\sin(s)\sinh(\alpha)) \\
e^t(\cos(s)\cosh(\alpha)-\sin(s)\sinh(\alpha)) & e^t(\sin(s)\cosh(\alpha)-\cos(s)\sinh(\alpha))
\end{pmatrix}
\end{align*}
\end{example}
This example permits a similar lemma to Lemma \ref{lem:TypeI}, but applicable to the case when the immersions $p$ and $q$ are of type $IV$.
\begin{lemma}\label{lem:TypeIV}
Let $F:\Sigma\rightarrow\SL:(x, y)\mapsto(p(x, y), q(x, y))$ be an immersion of a degenerate almost complex surface, for which the almost product structure $P$ does not preserve the tangent bundle and let $\phi$ be the constant angle function locally determining $F$. If the immersion $p$ is of type $IV$, then it is congruent to an open part of the immersion given in Example \ref{example:IV}, with congruency condition
\begin{align*}
3\tanh^2(2\alpha)&=4\sin^2\psi,
\end{align*}
where $\psi=\phi+\frac{\pi}{3}$. A completely analogous statement holds for when immersion $q$ is of type $IV$, when one replaces the angle function $\psi$ by $\xi=\pi-\frac{\pi}{3}$. 
\end{lemma}
\begin{proof}
As before, we will once again only focus on the immersion $p:\Sigma\rightarrow\Sl$, as the proof for the immersion $q$ is completely analogous. A similar argument as in the proof of Lemma \ref{lem:TypeI} immediately shows that the if the immersion $p$ is of type $IV$, that it is congruent to the surface given in Example \ref{example:IV}, with congruency condition 
\begin{align*}
3\tanh^2(2\alpha)&=4\sin^2\psi,
\end{align*}
where $\psi$ is the constant angle function $\psi=\phi+\frac{\pi}{3}$. One immediately obtains the same result for the immersion $q$, when one substitutes the angle $\psi$ by the constant angle function $\xi$, defined as $\xi=\pi-\frac{\pi}{3}$.
\end{proof}

\section{Degenerate almost complex surfaces with four-dimensional distribution $\mathcal{D}$}\label{section:classifyingresult}
In this section, we will first use the isoparametric surfaces in $\SL$, as described in Section \ref{section:Isoparam}, to construct surfaces in the ambient nearly K\"ahler $\SL$. Subsequently, we will prove that these surfaces, satisfying some additional conditions, encompass all the degenerate almost complex surfaces in $\SL$ for which the distribution $\mathcal{D}$, defined in Equation (\ref{def: distrD}), is four-dimensional. This is demonstrated in Theorem \ref{theorem:classifying result}, thereby classifying all degenerate almost complex surfaces for which the almost product structure $P$ does not preserve the tangent bundle.

\subsection{Examples of degenerate almost complex surfaces}

In this section, we construct new surfaces in the nearly K\"ahler $\SL$ by considering the assumption that the surface restricted to one of the factors $\SL$ is a Lorentzian isoparametric surface in $\Sl$, which are classified in Section \ref{section:Isoparam}. We define these surfaces as $F:\mathbb{R}^2\rightarrow\SL:(x, y)\mapsto(f_1(x, y), f_2(x, y))$, where $f_1, f_2:\mathbb{R}^2\rightarrow\Sl$ are Lorentzian isoparametric surfaces.

\begin{remark}
Initially, we consider two immersions, $f_1$ and $f_2$, which are parameterized by coordinates $\left\{s,t\right\}$ and $\left\{\tilde{s}, \tilde{t}\right\}$, respectively. To describe the immersion $F$ in $\SL$, we introduce normalized null coordinates $\left\{u,v\right\}$ and $\left\{\tilde{u},\tilde{v}\right\}$ on these surfaces, as exemplified in Section \ref{section:Isoparam}. We also use the compatibility conditions outlined in the proof of Lemma \ref{lem:null coord}. For the sake of clarity, we will first express the immersions $f_1$ and $f_2$ using the coordinates $\left\{s,t\right\}$ and $\left\{\tilde{s}, \tilde{t}\right\}$, respectively. Then, we'll provide the compatibility equations using the null coordinates.

Note that when the immersion $f_1$ is of type $IV$, we will follow the notation introduced in Example \ref{example:IV}. Similarly, when dealing with the immersion $f_2$ of type $IV$, we will denote the parameter $\alpha$ as $\beta$.
\end{remark}
We first present an example of a surface, where the immersions $f_1$ is of type $II_a$ and the immersion $f_2$ is of type $II_d$, thus both of them are described by Lemma \ref{lem:TypeII}.
\begin{example}[Immersion of type $II_a-II_d$]\label{example:II-IIa}
	Consider the immersion \\${F:\mathbb{R}^2\rightarrow\SL: (u,v)\mapsto(f_1(u,v), f_2(u, v))}$, with expressions
	\begin{align*}
		f_1(u,v)&= \begin{pmatrix}
			\frac{(1-u+2v)}{2}e^u & -\frac{(1+u-2v)}{2}e^{-u} \\
			e^u & e^{-u}
		\end{pmatrix},\\
		f_2(\tilde{u},\tilde{v})&=\begin{pmatrix}
			\frac{-\sin(\tilde{u})+(\tilde{u}+2\tilde{v})\cos(\tilde{u})}{2} & \frac{\cos(\tilde{u})+(\tilde{u}+2\tilde{v})\sin(\tilde{u})}{2}\\
			-2\cos(\tilde{u}) & -2\sin(\tilde{u})
		\end{pmatrix},
	\end{align*}
	where $\tilde{u}=-\frac{1}{2}u-v$ and $\tilde{v}=\frac{3}{4}u-\frac{1}{2}v$. 
\end{example}
We can after a rather lengthy calculation observe that this surface is, in fact, a degenerate almost complex surface in the nearly K\"ahler $\SL$. Now, we will introduce surfaces where the immersion $f_1$ belongs to type $II$, and the immersion $f_2$ belongs to type $IV$, with the added condition that the parameter $\alpha$ in Example \ref{example:IV} is zero. These four surfaces also serve as examples of degenerate almost complex surfaces, as a lengthy but straightforward calculation confirms.
\begin{example}[Immersion of type $II_a-IV$]\label{example:II-IV}
	Consider the immersions \\${F_i:\mathbb{R}^2\rightarrow\SL: (u,v)\mapsto (f_{1i}(u,v), f_2(u, v))}$, with $i\in\left\{a, b, c, d\right\}$. These immersions are then expressed as
	\begin{alignat*}{5}
		&f_{1a}(u,v)= \begin{pmatrix}
			\frac{(1-u+2v)}{2}e^u & -\frac{(1+u-2v)}{2}e^{-u} \\
			e^u & e^{-u}
		\end{pmatrix},\quad{}
		&&f_{1c}(u,v)= \begin{pmatrix}
			\frac{-\sin(v)+(v+2u)\cos(v)}{2} & \frac{\cos(v)+(v+2u)\sin(v)}{2}\\
			-2\cos(v) & -2\sin(v)
		\end{pmatrix},\\
		&f_{1b}(u,v)= \begin{pmatrix}
			\frac{(1-v+2u)}{2}e^v & -\frac{(1+v-2u)}{2}e^{-v} \\
			e^v & e^{-v}
		\end{pmatrix},\quad{}
		&&f_{1d}(u,v)= \begin{pmatrix}
			\frac{-\sin(u)+(u+2v)\cos(u)}{2} & \frac{\cos(u)+(u+2v)\sin(u)}{2}\\
			-2\cos(u) & -2\sin(u)
		\end{pmatrix},\\
		&f_2(\tilde{s}, \tilde{t})=\begin{pmatrix}
			e^{-\tilde{t}}\sin(\tilde{s}) &-e^{-\tilde{t}}\cos(\tilde{s}) \\
			e^{\tilde{t}}\cos(\tilde{s}) & e^{\tilde{t}}\sin(\tilde{s})
		\end{pmatrix}, 
	\end{alignat*}
	where the coordinates $\left\{\tilde{s},\tilde{t} \right\}$ are given by 
	\begin{alignat*}{5}
		&\tilde{s}=\frac{\sqrt{2}}{2}(v-u)\;\quad{}&&
		\tilde{t}=\frac{\sqrt{2}}{2}(u+v),\\
		&\tilde{u}=-\frac{\sqrt[4]{3}}{2}u-\sqrt[4]{3}v,\;\quad{}&&\tilde{v}=\frac{3^{3/4}}{4} u -\frac{1}{2\sqrt[4]{3}}v.
	\end{alignat*} 
\end{example}

The following two examples are surfaces where the immersion $f_1$ is of type $I$ and the immersion $f_2$ is of type $IV$. Note that there are two examples as a Lorentzian isoparametric surface of type $I$ can have two forms, as illustrated by Lemma \ref{lem:TypeI}, while Lemma \ref{lem:TypeIV} yields the description of $f_2$. These immersions are not automatically degenerate almost complex surfaces, but in Section \ref{section:classifyingresult} we will present sufficient conditions on the parameters $\lambda$ and $\beta$ for these surfaces to be degenerate almost complex surfaces. 
\begin{example}[Immersion of type $I-IV$]\label{example:IaIV}
	Consider the immersion \\${F:\mathbb{R}^2\rightarrow\SL: (u,v)\mapsto (f_1(u,v), f_2(u, v))}$, with expressions
	\begin{align*}
		f_1(s,t)&=\sqrt{\frac{1}{1-\lambda^2}}\begin{pmatrix}
			\cos(s)-\sqrt{\lambda^2}\cos(t) & -\sin(s)+\sqrt{\lambda^2}\sin(t) \\
			\sin(s)+\sqrt{\lambda^2}\sin(t)  & \cos(s)-\sqrt{\lambda^2}\cos(t)
		\end{pmatrix},\\
		f_2(\tilde{s},\tilde{t})&=\begin{pmatrix}
			e^{-\tilde{t}}(\sin(\tilde{s})\cosh(\beta)-\cos(\tilde{s})\sinh(\beta)) &-e^{-\tilde{t}}(\cos(\tilde{s})\cosh(\beta)-\sin(\tilde{s})\sinh(\beta)) \\
			e^{\tilde{t}}(\cos(\tilde{s})\cosh(\beta)-\sin(\tilde{s})\sinh(\beta)) & e^{\tilde{t}}(\sin(\tilde{s})\cosh(\beta)-\cos(\tilde{s})\sinh(\beta))
		\end{pmatrix},
	\end{align*}
	with $\lambda,\beta\in\mathbb{R}$ and $\lambda^2<1$. The coordinates $\left\{s, t \right\}$ and $\left\{\tilde{s}, \tilde{t}\right\}$ are given by
	\begin{alignat*}{5}
		&s=-r\sqrt{1-\lambda^2}u+\frac{1}{2r}\sqrt{1-\lambda^2}v,\;\quad{}&&
		t=r\sqrt{\frac{1-\lambda^2}{\lambda^2}}u+\frac{1}{2r}\sqrt{\frac{1-\lambda^2}{\lambda^2}}v\\
		&\tilde{s}=-(R+R\tanh(2\beta))\tilde{u}+\frac{1}{2R}(1-\tanh(2\beta)\tilde{v},\;\quad{}&&
		\tilde{t}=R\sech(2\beta)\tilde{u}+\frac{1}{2R}\sech(2\beta)\tilde{v}\\
		&\tilde{u}=-\frac{r}{2R}u-\frac{\sqrt{3}}{2rR}v,\;\quad{}&&\tilde{v}=\frac{\sqrt{3}rR}{2}u-\frac{R}{2r}v, 
	\end{alignat*}
	where $r$ and $R$ are fixed to make $\left\{u,v\right\}$ and $\left\{\tilde{u},\tilde{v}\right\}$ normalized null coordinates by
	\begin{align*}
		(1-\lambda^2)^2r^4=\lambda^2,\quad{}4R^4=\cosh^2(2\beta)e^{-4\beta}.
	\end{align*} 
\end{example} 
\begin{example}[Immersion of type $I-IV$]\label{example:IbIV}
	Consider the immersion \\${F:\mathbb{R}^2\rightarrow\SL: (u,v)\mapsto (f_1(u,v), f_2(u, v))}$, with expressions
	\begin{align*}
		f_1(s,t)&=\sqrt{\frac{1}{\lambda^2-1}}\begin{pmatrix}
			-e^{-s} & -\sqrt{\lambda^2}e^{-t}\\
			\sqrt{\lambda^2}e^t & e^s
		\end{pmatrix},\\
		f_2(\tilde{s},\tilde{t})&=\begin{pmatrix}
			e^{-\tilde{t}}(\sin(\tilde{s})\cosh(\beta)-\cos(\tilde{s})\sinh(\beta)) &-e^{-\tilde{t}}(\cos(\tilde{s})\cosh(\beta)-\sin(\tilde{s})\sinh(\beta)) \\
			e^{\tilde{t}}(\cos(\tilde{s})\cosh(\beta)-\sin(\tilde{s})\sinh(\beta)) & e^{\tilde{t}}(\sin(\tilde{s})\cosh(\beta)-\cos(\tilde{s})\sinh(\beta))
		\end{pmatrix},
	\end{align*}
	with $\lambda,\beta\in\mathbb{R}$ and $\lambda>1$. The coordinates $\left\{s, t \right\}$ and $\left\{\tilde{s}, \tilde{t}\right\}$ are given by
	\begin{alignat*}{5}
		&s=-r\sqrt{\lambda^2-1}u+\frac{1}{2r}\sqrt{\lambda^2-1}v,\;\quad{}&&
		t=r\sqrt{\frac{\lambda^2-1}{\lambda^2}}u+\frac{1}{2r}\sqrt{\frac{\lambda^2-1}{\lambda^2}}v\\
		&\tilde{s}=-(R+R\tanh(2\beta))\tilde{u}+\frac{1}{2R}(1-\tanh(2\beta)\tilde{v},\;\quad{}&&
		\tilde{t}=R\sech(2\beta)\tilde{u}+\frac{1}{2R}\sech(2\beta)\tilde{v}\\
		&\tilde{u}=-\frac{r}{2R}u-\frac{\sqrt{3}}{2rR}v,\;\quad{}&&\tilde{v}=\frac{\sqrt{3}rR}{2}u-\frac{R}{2r}v, 
	\end{alignat*}
	where $r$ and $R$ are fixed to make $\left\{u,v\right\}$ and $\left\{\tilde{u},\tilde{v}\right\}$ normalized null coordinates by
	\begin{align*}
		(1-\lambda^2)^2r^4=\lambda^2,\quad{}4R^4=\cosh^2(2\beta)e^{-4\beta}.
	\end{align*}
\end{example}
We conclude our list of examples by considering the case where the immersions $f_1$ and $f_2$ are both of type $IV$, which are classified in Lemma \ref{lem:TypeIV}. Section \ref{section:SL} will also introduce conditions on the parameters $\alpha$ and $\beta$ for these surfaces to be degenerate almost complex surfaces.
\begin{example}[Immersion of type $IV-IV$]\label{example:IVIV}
	Consider the immersion \\${F:\mathbb{R}^2\rightarrow \SL:(u,v)\mapsto (f_1(u,v), f_2(u,v))}$, with expressions
	\begin{align*}
		f_1(s,t)&=\begin{pmatrix}
			e^{-t}(\sin(s)\cosh(\alpha)-\cos(s)\sinh(\alpha)) &-e^{-t}(\cos(s)\cosh(\alpha)-\sin(s)\sinh(\alpha)) \\
			e^t(\cos(s)\cosh(\alpha)-\sin(s)\sinh(\alpha)) & e^t(\sin(s)\cosh(\alpha)-\cos(s)\sinh(\alpha))
		\end{pmatrix},\\
		f_2(\tilde{s}, \tilde{t})&=\begin{pmatrix}
			e^{-\tilde{t}}(\sin(\tilde{s})\cosh(\beta)-\cos(\tilde{s})\sinh(\beta)) &-e^{-\tilde{t}}(\cos(\tilde{s})\cosh(\beta)-\sin(\tilde{s})\sinh(\beta)) \\
			e^{\tilde{t}}(\cos(\tilde{s})\cosh(\beta)-\sin(\tilde{s})\sinh(\beta)) & e^{\tilde{t}}(\sin(\tilde{s})\cosh(\beta)-\cos(\tilde{s})\sinh(\beta))
		\end{pmatrix},
	\end{align*}
with $\alpha, \beta\in\mathbb{R}$. The coordinates $\left\{s, t\right\}$ and $\left\{ \tilde{s}, \tilde{t} \right\}$ are given by
	\begin{align*}
		s&=-(r+r\tanh(2\alpha))u+\frac{1}{2r}(1-\tanh(2\alpha)v,\quad{}
		t=r\sech(2\alpha)u+\frac{1}{2r}\sech(2\alpha)v,\\
		\tilde{s}&=-(R+R\tanh(2\beta))\tilde{u}+\frac{1}{2R}(1-\tanh(2\beta))\tilde{v},\quad{}
		\tilde{t}=R\sech(2\beta)\tilde{u}+\frac{1}{2R}\sech(2\beta)\tilde{v},
	\end{align*}
	where $r$ and $R$ are fixed to make $\left\{u, v\right\}$ and $\left\{\tilde{u}, \tilde{v} \right\}$ normalized null coordinates by
	\begin{align*}
		4r^4&=\cosh^2(2\alpha)e^{-4\alpha},\quad{} 4R^4=\cosh^2(2\beta)e^{-4\beta}.
	\end{align*}
\end{example}
\subsection{Classifying result}

In this section we will show that all degenerate almost complex surfaces in the nearly K\"ahler $\SL$ for which the almost product structure $P$ does not preserve the tangent bundle, are in fact locally congruent to one of the examples in the previous section, when subject to some additional conditions. This is demonstrated in \ref{theorem:classifying result}. 
\begin{named}{Main Theorem 3}\label{theorem:classifying result}
	Let $\Sigma$ be an immersed, degenerate almost complex surface in the nearly K\"ahler \\${(\SL, J, g)}$ for which the almost product structure $P$ does not preserve the tangent bundle. Denote the constant angle function locally determining the immersion  as $\phi$ and define two constant angles $\psi, \xi$ as $\psi=\phi+\frac{\pi}{3}$ and $\xi=\phi-\frac{\pi}{3}$. Then $\Sigma$ is locally congruent to an open part of the immersion $F:\mathbb{R}^2\rightarrow\SL$, whose expression depends on the value of $\phi$.
	\begin{align*}
	&\text{If } \phi\in\left\{0, \pi \right\}, \text{ then } F \text{ is described in Example \ref{example:II-IIa}};\\
	&\text{If }\phi\in\left\{\frac{\pi}{3},\frac{2\pi}{3},\frac{4\pi}{3},\frac{5\pi}{3}\right\},  \text{ then } F \text{ is described by immersion} \\&\quad{}F_b, F_a, F_d\text{ and }F_c,\text{ respectively, in Example \ref{example:II-IV}};\\		
	&\text{If } \phi\in\left]0, \frac{\pi}{3} \right[\cup\left]\pi, \frac{4\pi}{3} \right[,\text{ then }F \text{ is described in Example \ref{example:IaIV} or Example \ref{example:IbIV}},\\&\quad{} \text{with extra conditions that } 3(1+\lambda^2)^2=\lambda^2\sin^2\psi\text{ and }3\tanh^2(2\beta)=4\sin^2\xi;\\
	&\text{If } \phi\in\left]\frac{\pi}{3}, \frac{2\pi}{3}\right[\cup\left]\frac{4\pi}{3}, \frac{5\pi}{3}\right[, \text{ then } F \text{ is described in Example \ref{example:IVIV}},\\&\quad{} \text{with extra conditions that }3\tanh^2(2\alpha)=4\sin^2\psi\text{ and }3\tanh^2(2\beta)=4\sin^2\xi;\\
	&\text{If } \phi\in\left]\frac{2\pi}{3}, \pi\right[\cup\left]\frac{5\pi}{3}, 2\pi\right[, \text{ then } F \text{ is described in Example \ref{example:IaIV} or Example \ref{example:IbIV}},\\&\quad{} \text{with extra conditions that } 3(1+\lambda^2)^2=\lambda^2\sin^2\xi\text{ and }3\tanh^2(2\beta)=4\sin^2\psi.
\end{align*}
\end{named}
\begin{proof}
Let $F:\Sigma\rightarrow\SL:(x, y)\mapsto (p(x, y), q(x, y))$ be an immersion of a degenerate almost complex surface in the nearly K\"ahler $\SL$, for which the almost product structure $P$ does not preserve the tangent bundle. \ref{Main theorem 2} then shows that this immersion is locally determined by a constant angle function $\phi$. Proposition \ref{prop:Isoparam} yields that the immersions $p,q:\Sigma\rightarrow\Sl$ are flat Lorentzian isoparametric surfaces in the semi-Riemannian $\Sl$, where the length of their mean curvature vector fields depend on the constant angle function $\phi$ as
\begin{align}\label{eq:Hproof}
\text{immersion }p\,:\, \left\langle H, H\right\rangle\,=\frac{4\sin^2\psi}{3},\quad{}\text{immersion }q\,:\, \left\langle H, H\right\rangle=\frac{4\sin^2\xi}{3},
\end{align}
where $\psi=\phi+\frac{\pi}{3}$ and $\xi=\phi-\frac{\pi}{3}$ are constant angle functions that were introduced for simplicity. Lemma \ref{lem:Classification of types} determines the possible types of these isoparametric immersions, which also depend on the angle $\phi$, as shown in Table \ref{table:types}. In Lemma \ref{lem:TypeI}, Lemma \ref{lem:TypeII} and Lemma \ref{lem:TypeIV} we give classification results for the immersions $p$ and $q$, where Equation (\ref{eq:Hproof}) plays an important role in ensuring congruency. Note that these results do not automatically ensure that the immersion $F:\Sigma\rightarrow\SL$ is a degenerate almost complex surface, as the coordinates of the parametrizations of $p$ and $q$ do not necessarily coincide. This issue can be resolved by applying Lemma \ref{lem:null coord}, which showed how to construct generic null coordinates $\left\{u,v\right\}$ and $\left\{\tilde{u}, \tilde{v} \right\}$ on $p$ and $q$, respectively, with compatibility conditions
\begin{align*}
\tilde{u}&=-\frac{r}{2R}u-\frac{\sqrt{3}}{2rR}v,\quad{}\tilde{v}=\frac{\sqrt{3}rR}{2}u-\frac{R}{2r}v, 
\end{align*} 
where the values of the constants $r$ and $R$ depend on the constant angle function $\phi$. We will now classify all the degenerate almost complex surfaces in $\SL$ and distinguish between the different values of the constant angle $\phi$, as shown in Table \ref{table:types}.

\subsubsection*{\textbf{Angle} $\mathbf{\phi}$ \textbf{is zero}} In this case the angle $\psi$ is equal to $\frac{\pi}{3}$ and the angle $\xi$ is equal to $\frac{5\pi}{3}$. Lemma \ref{lem:Classification of types} then shows that both immersions are of type $II$. More specifically, one can deduce that $p$ is of type $II_a$, while $q$ is of type $II_d$ by applying Lemma \ref{lem:TypeII}. 
Compatibility conditions for the coordinates on the immersions $p$ and $q$ are given by Lemma \ref{lem:null coord}, thereby ensuring that the surface $\Sigma$ is locally congruent to an open part of the immersion $F$ in Example \ref{example:II-IIa}. Note that this immersion is a degenerate almost complex surface, as stated in the example.

\subsubsection*{\textbf{Angle} $\mathbf{\phi}$ \textbf{is between zero and} $\mathbf{\frac{\pi}{3}}$}
Lemma \ref{lem:Classification of types} then implies the immersion $p$ is of type $I$, while the immersion $q$ is of type $IV$, with constant angles $\psi\in\left]\frac{\pi}{3},\frac{2\pi}{3}\right[$ and $\xi\in\left]\frac{5\pi}{3}, 2\pi\right[$. Lemma \ref{lem:TypeI} shows that $p$ is congruent to an open part of either the immersion in Example \ref{example:Ia} or the immersion in Example \ref{example:Ib}, with congruency condition $3(1+\lambda^2)^2=\lambda^2\sin^2\psi$. Similarly, Lemma Lemma \ref{lem:TypeIV} yields that the immersion $q$ is congruent to an open part of the immersion in Example \ref{example:IV}, with congruency condition $3\tanh^2(2\beta)=4\sin^2\xi$. Introducing compatible null coordinates on these immersions, as described in Lemma \ref{lem:null coord} then shows that the surface $\Sigma$ is congruent to an open part of the immersion defined in Example \ref{example:IaIV} or in Example \ref{example:IbIV}. A straightforward, but tedious calculation shows that these immersions, satisfying the congruency conditions, are indeed degenerate almost complex surfaces in the nearly K\"ahler $\SL$.

\subsubsection*{\textbf{Angle} $\mathbf{\phi}$ \textbf{is equal to} $\mathbf{\frac{\pi}{3}}$} The constant angles $\psi$ and $\xi$ now satisfy $\psi=\frac{2\pi}{3}$ and $\xi=0$. Lemma \ref{lem:Classification of types} thus implies that the immersion $p$ is of type $II$, while immersion $q$ is of type $IV$. Note that one can deduce that $p$ is actually of type $II_b$ from Lemma \ref{lem:TypeII} and that $q$ is congruent to an open part of the immersion in Example \ref{example:IV} with congruency condition $3\tanh^2(2\beta)=0$, due to Lemma \ref{lem:TypeIV}. Note that this implies the vanishing of the parameter $\beta$. Lemma \ref{lem:null coord} then once again gives compatibility conditions for the introduced null coordinates on both immersions, showing that $\Sigma$ has to be congruent to an open part of the immersion $F_b$ in Example \ref{example:II-IV}, which defines an degenerate almost complex surface.

\subsubsection*{\textbf{Angle} $\mathbf{\phi}$ \textbf{is between} $\mathbf{\frac{\pi}{3}}$ \textbf{and} $\mathbf{\frac{2\pi}{3}}$}
In this case both the immersion $p$ and the immersion $q$ are of type $IV$, with constant angles $\psi\in\left]\frac{2\pi}{3}, \pi\right[$ and $\xi\in\left]0,\frac{\pi}{3}\right[$, as shown in Lemma \ref{lem:Classification of types}. Lemma \ref{lem:TypeIV} and Lemma \ref{lem:null coord} then show, in a similar way as the previous cases, that the surface $\Sigma$ is locally congruent to an open part of the immersion in Example \ref{example:IVIV}, with additional congruency conditions $3\tanh^2(2\alpha)=4\sin^2\psi$ and $3\tanh^2(2\beta)=4\sin^2\xi$. One can, after a rather lengthy computation, also check that these conditions make this immersion a degenerate almost complex surface.

\subsubsection*{\textbf{Angle} $\mathbf{\phi}$ \textbf{is equal to} $\mathbf{\frac{2\pi}{3}}$} Table \ref{table:types} shows that the immersion $p$ is now of type $IV$, while the immersion $q$ is of type $II$, with the angle $\psi$ equal to $\pi$ and the angle $\xi$ equal to $\frac{\pi}{3}$. Recall that the map 
\begin{align*}
\mathcal{F}_1:\SL\rightarrow\SL: (p,q)\mapsto(q,p),
\end{align*}
is an isometry, as shown in Lemma \ref{lem:isom}. Thus one can equivalently consider the case where $p$ is of type $II$, with respect to the angle $\xi=\frac{\pi}{3}$ and $q$ is of type $IV$, with respect to the angle $\psi=\pi$. Lemma \ref{lem:TypeII} then shows that the immersion $p$ is of type $II_a$, while Lemma \ref{lem:TypeIV} shows that $q$ is locally congruent to an open part of Example \ref{example:IV}, with the congruency condition implying that $\beta$ is once again zero. Completely analogous to the case where the angle $\phi$ was equal to zero, we can now deduce that the immersed surface $\Sigma$ is locally congruent to an open part of the immersion $F_a$ in Example \ref{example:II-IV}, which is automatically a degenerate almost complex surface.
\subsubsection*{\textbf{Angle} $\mathbf{\phi}$ \textbf{is between} $\mathbf{\frac{2\pi}{3}}$ \textbf{and} $\mathbf{\pi}$} Lemma \ref{lem:Classification of types} then immediately shows that the immersion $p$ is of type $IV$, while the immersion $q$ is of type $I$, with constant angle functions $\psi\in\left]\pi, \frac{4\pi}{3}\right[$ and $\xi\in\left]\frac{\pi}{3}, \frac{2\pi}{3}\right[$. We will once again use the isometry $\mathcal{F}_1$ to see that this case is equivalent with considering the immersion $p$ to be of type $I$, with respect to the angle  with angle $\xi$ between $\frac{\pi}{3}$ and $\frac{2\pi}{3}$ and immersion $q$ of type $IV$, with respect to the angle $\xi$ between $\pi$ and $\frac{4\pi}{3}$. Note that this is completely analogous to the case where the angle $\phi$ is between zero and $\pi$. Thus the surface $\Sigma$ is congruent to an open part of either the immersion in Example \ref{example:IaIV} or the one given in Example \ref{example:IbIV}, with congruency conditions $3(1+\lambda^2)^2=\lambda^2\sin^2\xi$ and $3\tanh^2(2\beta)=4\sin^2\xi$. Note that these congruency conditions ensure that these immersions are degenerate almost complex surfaces.
\subsubsection*{\textbf{Angle} $\mathbf{\phi}$ \textbf{is equal to} $\mathbf{\pi}$}
In this case constant angle functions satisfy  $\psi=\frac{4\pi}{3}$ and $\xi=\frac{2\pi}{3}$. Lemma \ref{lem:Classification of types} and Lemma \ref{lem:TypeII} imply that the immersion $p$ is then of type $II_c$ and the immersion $q$ is of type $II_b$. Using the isometry $\mathcal{F}_1$ one can also consider the equivalent case where $p$ is of type $II_b$ and $q$ is of type $II_c$. Lemma \ref{lem:TypeII} then shows that this is in turn also equivalent with the case that $p$ is of type $II_a$ and $q$ of type $II_d$ as this is simply an interchanging of coordinates on both immersions. This is exactly the case when the angle $\phi$ is equal to zero, thus the surface $\Sigma$ is locally congruent to an open part of the immersion in Example \ref{example:II-IIa}.
\subsubsection*{\textbf{Angle} $\mathbf{\phi}$ \textbf{is between} $\mathbf{\pi}$ \textbf{and} $\mathbf{\frac{4\pi}{3}}$}
We then immediately obtain that the immersion $p$ is of type $I$ and the immersion $q$ is of type $II$, with $\psi\in\left]\frac{4\pi}{3}, \frac{5\pi}{3} \right[$ and $\xi\in\left]\frac{2\pi}{3}, \pi \right[$.
The immersions $p$ is thus of type $I$ and the immersion $q$ of type $IV$, as shown in Table \ref{table:types}. This is exactly the same case as when the angle $\phi$ is between zero and $\frac{\pi}{3}$, thus $\Sigma$ is now locally congruent to an open part of either the immersion in Example \ref{example:IaIV}, or the immersion in Example \ref{example:IbIV}, with congruency conditions $3(1+\lambda^2)^2=\lambda^2\sin^2\psi$ and $3\tanh^2(2\beta)=4\sin^2\xi$. Thus once again these conditions make the immersion a degenerate almost complex surface in $\SL$, as a straightforward, but lengthy calculation can show.. 
\subsubsection*{\textbf{Angle} $\mathbf{\phi}$ \textbf{is equal to} $\mathbf{\frac{4\pi}{3}}$}
Lemma \ref{lem:Classification of types} and Lemma \ref{lem:TypeII} imply that the immersion $p$ is of type $II_d$ and the immersion $q$ is of type $IV$, with constant angle functions $\psi=\frac{5\pi}{3}$ and $\xi=\pi$. An analogous reasoning as in the case where the angle $\phi$ is equal to $\frac{\pi}{3}$, yields that $\Sigma$ is locally congruent to an open part of the immersion $F_d$ in Example \ref{example:II-IV}.
\subsubsection*{\textbf{Angle} $\mathbf{\phi}$ \textbf{is between} $\mathbf{\frac{4\pi}{3}}$ \textbf{and} $\mathbf{\frac{5\pi}{3}}$}
Table \ref{table:types} shows that the immersions $p$ and $q$ are now both of type $IV$, with angle functions $\psi\in\left]\frac{5\pi}{3},2\pi \right[$ and $\xi\in\left]\pi,\frac{4\pi}{3} \right[$. Note that this is similar to the case where the angle function $\phi$ is between $\frac{\pi}{3}$ and $\frac{2\pi}{3}$. Thus we can conclude that the surface $\Sigma$ is locally congruent to an open part of the immersion in Example \ref{example:IVIV}, with congruency conditions $3\tanh^2(2\alpha)=4\sin^2\psi$ and $3\tanh^2(2\beta)=4\sin^2\xi$, making this immersions once again an almost complex surface.

\subsubsection*{\textbf{Angle} $\mathbf{\phi}$ \textbf{is equal to} $\mathbf{\frac{5\pi}{3}}$}
Lemma \ref{lem:Classification of types} and Lemma \ref{lem:TypeII} imply that the immersion $p$ in this case is of type $IV$, while the immersion $q$ is of type $II_c$, with the angle $\psi$ equal to zero and the angle $\xi$ equal to $\frac{4\pi}{3}$. Applying the isometry $\mathcal{F}_1$ shows than we can consider the equivalent case where the immersion $p$ is of type $II_c$ and the immersion $q$ is of type $IV$, with respect to the angle $\psi=0$. Analogously to the case where the angle $\phi=\frac{\pi}{3}$, one can deduce that $\Sigma$ is locally congruent to an open part of the immersion $F_c$ in Example \ref{example:II-IV}. 
\subsubsection*{\textbf{Angle} $\mathbf{\phi}$ \textbf{is between} $\mathbf{\frac{5\pi}{3}}$ \textbf{and} $\mathbf{2\pi}$}
In this final case on has that the immersion $p$ is of type $IV$ and the immersion $q$ is of type $I$, with angles $\psi\in\left]0,\frac{\pi}{3} \right[$ and $\xi\in\left]\frac{5\pi}{3}, 2\pi \right[$. Remark that this case is equivalent to the case where the angle $\phi$ is between $\frac{2\pi}{3}$ and $\pi$. We can thus conclude that the surface $\Sigma$ is congruent to an open part of either the immersion in Example \ref{example:IaIV} or the one given in Example \ref{example:IbIV}, with congruency conditions $3(1+\lambda^2)^2=\lambda^2\sin^2\xi$ and $3\tanh^2(2\beta)=4\sin^2\xi$. This immersion is finally also a degenerate almost complex surface, when one imposes these congruency conditions.

This concludes the proof of the theorem, as we have covered all possible values of the constant angle function $\phi$.
\end{proof}
\bibliographystyle{plain}
\bibliography{citationsSL2}
\end{document}